\documentclass[a4paper,11pt, fleqn]{article}

\usepackage[english]{babel}
\usepackage[utf8]{inputenc}

\usepackage{graphicx,amssymb}
\usepackage{textcomp}
\usepackage[bookmarks=true]{hyperref}
\usepackage{pdfsync}
\usepackage{latexsym}
\usepackage[small,bf]{caption}
\usepackage{amsfonts}
\usepackage{amsmath}
\usepackage{stmaryrd}
\usepackage{bbm}
\usepackage{theorem}
\usepackage{times}
\usepackage{color}
\usepackage{hyperref}
\newcommand{\bd}{\begin{displaymath}}
\newcommand{\be}{\begin{equation}}
\newcommand{\beq}{\begin{eqnarray}}
\newcommand{\ba}{\begin{array}}
\newcommand{\ed}{\end{displaymath}}
\newcommand{\ee}{\end{equation}}
\newcommand{\eeq}{\end{eqnarray}}
\newcommand{\ea}{\end{array}}
\newcommand{\espace}{\mbox{ }}

\newcommand{\bb}{\mathbb}
\newcommand{\N}{{\mathbb N}}

\newcommand{\Z}{{\mathbb Z}}

\newcommand{\R}{{\mathbb R}}
\newcommand{\PP}{{P}}
\newcommand{\EE}{{E}}

\newtheorem{theorem}{Theorem}[section]

\newtheorem{proposition}{Proposition}[section]

\newtheorem{lemma}{Lemma}[section]
\newtheorem{corollary}{Corollary}[section]

\newtheorem{remark}{Remark}[section]
\newenvironment{proof}[2]{\espace\\{\em Proof of #1 \ref{#2}.}}{\hfill\mbox{$\square$}}
\newenvironment{proofmilton}{\espace{\em Proof.}}{\hfill\mbox{$\square$}}

\numberwithin{equation}{section}

\title{Zero-range processes with rapidly growing rates}

\author{ 
Enrique Andjel\thanks{Aix Marseille Universit\'e, CNRS, Centrale Marseille I2M UMR 7373 13453, Marseille, France. Visiting IMPA. Email: {\tt andjel@impa.br}},
In\'es Armend\'ariz\thanks{IMAS-CONICET and Departamento de Matem\'atica, FCEN, Universidad de Buenos Aires. Pabell\'on 1, Ciudad Universitaria, 1428 Buenos Aires, Argentina . Email: {\tt iarmend@dm.uba.ar}},  Milton Jara \thanks{ Instituto de Matem\'atica Pura e Aplicada, IMPA. Estrada Dona Castorina 110, 22460- 320, Rio de Janeiro, Rio de Janeiro, Brasil. Email: {\tt monets@impa.br}} 
}

\begin{document}
\sloppy
\maketitle
\begin{abstract}
\noindent
We provide two methods to construct zero-range processes with superlinear rates on  $\Z^d$. In the first method  these rates can 
grow very fast, if either the dynamics and the initial distribution are translation invariant or if only nearest neighbour translation invariant jumps are permitted, in the one-dimensional lattice. In the second method the rates cannot grow as fast but more general dynamics are allowed.
\end{abstract}

\section{Introduction}\label{sec:intro}

The zero-range process was introduced by Spitzer \cite{S} as a Markov process on a $\N_0^S$, 
where $\N_0$ is the set of non-negative integers, and S is a denumerable set. In this process particles are indistinguishable, 
and a particle leaves a given site $x$ at rate $g(n)$, where $n$ is the number of particles present at $x$. 
Once a particle jumps from $x$ it moves to a site $y$ chosen according to a transition probability matrix $p(x,y)$ on $S$. 
The choice of target site $y$ is independent of the time at which the jump occurs and of the past of the process. Throughout this paper $S$ will be an integer lattice and $p(x,y)$ will be the transition matrix of a random walk on
that lattice. On occasions we will write $p(z)$ for $p(0,z)$.

The existence of the dynamics was proved initially by Holley \cite{H} and Liggett \cite{L}.   
Their results were extended by   Andjel \cite{A} who adapted to the zero range process a technique 
introduced by Liggett and Spitzer \cite{LS}. Andjel assumes that the rates satisfy a 
Lipschitz condition $\sup_{n \ge 0}|g(n+1)-g(n)|<\infty$, thus imposing that the rates grow at most linearly.  
 More recently, Bal\'azs, Rassoul-Agha, Sepp\"al\"ainen and Sethuraman \cite{BRSS} construct 
 the zero-range process with totally asymmetric dynamics $\displaystyle{p(x,y)=\mathbf{1}(y-x=1)}$ and nearest neighbour jumps in 
 the one-dimensional lattice $\Z$, under the assumption that the jump rates are non-decreasing and grow at most exponentially. 
 Under these conditions, they prove that the process is Markov and admits a one parameter  family of extremal invariant measures. 
 Their proofs are based on a representation of the model as a system of columns with monotonically increasing heights, for which
 the totally asymmetric assumption on the  dynamics is crucial. 

 In this article we introduce two methods to construct zero-range processes with superlinear rates on integer lattices, 
and identify the associated martingales. The first method allows for quite general rate functions $g$, 
but requires either nearest neighbour transition probabilities   on the one-dimensional lattice $\Z$, or that both the dynamics
and the initial distribution be translation invariant on $\Z^d$.
The second method can be applied to quite general random walks on $\Z^d$, but is more restrictive on the rate functions. 

\section{Notation and results\label{n&r}}

Throughout the article the set of sites will be the integer lattice, $S=\Z^d$. Given a transition matrix $p(\cdot, \cdot)$ on $\Z^d$ and rate function 
$g:\N_0\to [0, +\infty)$ such that $g(0)=0$, our goal is to construct an associated Markov process on the state space 
\begin{equation*}
X:=\N_0^{\Z^d}  
\end{equation*}
endowed with the product topology. Elements $\eta \in X$ will be called configurations, with $\eta=\big(\eta(x):\, x\in \Z^d \big)$, $\eta(x) \in \N_0,$ the number of particles at site $x$. We also define 
\begin{equation}
\label{finite}
X_f:=\big\{\eta\in X, \sum_x \eta(x)<\infty \big\}
\end{equation}
the set of configurations with finitely many particles.
We say that a function $f:X\to \R$ is local if there exists a finite set $A\subset \Z^d$ such that 
$f(\eta)=f(\xi)$ whenever $\eta(x)=\xi(x)$ for all $x\in A$. We call the smallest such set the support of $f$.

Let $g:\N\cup\{0\}\to \R_{\ge 0},\,g(0)=0$. The  formal generator of our dynamics is 
\begin{equation}
\label{generator}
Lf(\eta)=\sum_{x,y\in \Z^d,x\neq y}g(\eta (x))p(x,y)\big(f(\eta^{x,y})-f(\eta)\big),
\end{equation}
where $f:X\to \R$ is a bounded local function, and 
\begin{equation}
\label{eta-xy}
\eta^{x,y}(z)=\begin{cases}\eta(x)-1, &z=x \text{ and } \eta(x)\ge 1, \\ \eta(y)+1 & z=y \text{ and } \eta(x)\ge 1, \\ \eta(z) & \text{ otherwise.} \end{cases}
\end{equation} 
Informally, at rate $g(k)$ a site $x$ containing $k$ particle loses one that jumps to site $y$ with probability $p(x,y)$. We will say that a process $(\eta_t,\,t\geq 0)$  on a subset of $X$ is
a solution of the martingale problem associated to $L$ if for any local bounded function $f$
$$f(\eta_t) -f(\eta_0)-\int_0^tLf(\eta_s)ds, \qquad t\geq 0$$
is a martingale. We also say that the process satisfies the integrated forward equation if for any $f$ as above
\begin{equation}
\label{integrated}
E\big[f(\eta_t)\big]= f(\eta_0) + \int_0^t E\big[L f(\eta_u)\big] du \qquad  \forall t\geq 0.
\end{equation}
In some cases we can show a stronger result, the forward equation, that is:
\begin{equation}
\label{forward}
\frac{d}{dt}E^{\eta_0}[f(\eta_t)]=E^{\eta_0}\big[Lf(\eta_t)\big].
\end{equation}

We are interested in the situation when the rates are non-decreasing and diverge at $\infty$,
\begin{equation}
\label{crecientes} 
 g(n)\le g(n+1),\,\ \  n\in \N_0, \quad \text{and} \quad \lim_{n\to \infty} g(n)=\infty. 
\end{equation}
Condition \eqref{crecientes} will imply that the processes we construct are attractive, that is, the coordinate-wise partial order of configurations
\[
\eta, \xi \in X,\, \eta\le \xi \iff \eta(x)\le \xi(x)\, \forall\,  x\in \Z^d
\]
 is preserved by the dynamics. This means that there exists a coupled process $\big((\eta_t,\xi_t),\, t\ge 0\big)$ with initial value $(\eta,\xi)$ such that $P\big(\eta_t\le \xi_t \, \forall t\ge 0\big)=1$ and 
both $(\eta_t,\,t\ge 0)$ and $(\xi_t,\,t\ge 0)$ follow \eqref{generator}; the coupling in this case is said to be increasing. One such coupling is the basic coupling, which tries to match the two marginal processes as much as possible, and supplements the rates to get the right marginal distributions. Its generator is given by 
\begin{equation}
\label{basic}
\begin{split}
\hspace{-9mm}{\cal L}^{coup}f(\eta,\xi)&=\sum_{x,y\in \Z^d,\,x\neq y} g(\eta(x) \wedge \xi(x)) \,p(x,y)\big(f(\eta^{x,y}, \xi^{x,y})\big)-f(\eta,\xi)\big) \\
&\quad + \sum_{x,y\in \Z^d,\,x\neq y} [g(\eta(x))-g(\eta(x) \wedge \xi(x))]\,p(x,y) \big(f(\eta^{x,y}, \xi)\big)-f(\eta,\xi)\big) \\
&\quad + \sum_{x,y\in \Z^d,\,x\neq y} [g(\xi(x))- g(\eta(x) \wedge \xi(x))] \,p(x,y) \big(f(\eta, \xi^{x,y})\big)-f(\eta,\xi)\big),
\end{split}
\end{equation}
$f:X\times X\to \R$ a local, bounded function. This partial order on $X$ induces a partial order on the set of probability measures on $X$: given two such measures $\mu$ and $\nu$ we say that $\mu\leq \nu$ if
$$\int f(\eta)d\mu(\eta)\leq \int f(\eta)d\nu(\eta)$$
for any bounded local, increasing function $f$. By Strassen's Theorem, $\mu \leq \nu$ if and only if there exists a probability measure $\psi$ on $X \times X$ concentrating on $\{(\eta,\xi)\in X\times X: \eta\leq \xi\}$ whose first
and second marginals are $\mu$ and $\nu$ respectively.

The zero-range process started from an initial configuration  $\eta \in X_f$ is a well defined continuous time Markov process with bounded rates on a countable state space. We will denote by $S(t)$ the semigroup associated to $L$ acting on configurations with finitely many particles,
\begin{equation*}
S(t)f(\eta)=E^\eta[f(\eta_t)], \quad \eta \in X_f\/,f \text{ a bounded local function}.
\end{equation*}
The semigroup $S(t)$ also acts on probability measures on $X_f$: If $\mu$ is such a measure, then $\mu S(t) $    is the unique measure such that
\begin{equation*}
\int fS(t)(d\mu)=\int S(t)f d\mu \quad \forall \, \text{  bounded local function } f.
\end{equation*}
It then follows from the attractiveness of the process that for any bounded local function $f$,  $t\geq 0$, 
and $\eta\leq \xi \in X_f$, we have $ S(t)f(\eta)\leq S(t)f(\xi)$.
 And if $\mu\leq \nu$ are probability
 measures on $X_f$ then $\mu S(t)\leq \nu S(t)$.  In words, the semigroup $S(t)$ maps increasing functions into increasing functions and preserves the order of measures on $X_f$.

When the initial configuration $\eta\in X\setminus X_f$, we can consider an increasing sequence $\eta^n\to \xi$, $\eta^n \in X_f$ for all $n\ge 1$, and apply basic coupling to obtain a limiting process $\eta_t=\lim_{n\to \infty} \eta^n_t$, $t\ge 0$; this paper is concerned with finding conditions under which this process is well defined for all times, and identifying its martingales and invariant measures.

A set of transition probabilities $\{p(x,y)\}_{x,y \in \Z^d}$ is translation invariant when the random walk they determine is translation invariant; that is,  $p(x,y)=p(y-x)$ with $\{p(z)\}_{z\in \Z^d}$ such that $p(z)\ge 0, z\in \Z^d$, and $\sum_{z}p(z)=1$.
Let us now define a family $\{T^x: x\in \Z^d\}$ of translation operators acting on $X$, on $C(X)$ and on the set $\mathcal P$ of probability measures on $X$ as follows,
\begin{equation}
\label{to-1}
T^x(\eta)(y)=\eta(y-x), \ \  T^xf(\eta)=f(T^x\eta) \ \ \forall x,y\in\Z^d,\, \eta \in X,\, f\in C(X),
\end{equation}
and 
\begin{equation}
\label{to-2}
\int fd(T^x \mu)=\int (T^xf)d\mu  \ \ \forall x\in\Z^d,\, \mu \in \mathcal P,\, f\in C(X)\text{ bounded}.
\end{equation}
We say that $\mu \in \mathcal P$ is translation invariant if $T^x \mu =\mu$ for all $x\in \Z^d$.

The first result shows that the process can indeed be constructed starting from a translation invariant measure $\mu$ on $X$. In order to state it, we 
need to introduce a family of auxiliary measures  associated to $\mu$. Let us first define $[\mu]_n$ on $X$ as
\begin{equation}
\begin{split}
\label{mu_n}
\hspace{-10mm}&[\mu]_n\big(\eta: \eta(x_1)=k_1,\dots \eta(x_i)=k_i,\eta(y_1)=0,\dots,\eta(y_j)=0\big)\\
&\hspace{4.5cm}=\mu\big(\eta: \eta(x_1)=k_1,\dots \eta(x_i)=k_i\big),
\end{split}
\end{equation}
for all $i,j\in \N$, all $k_1,\dots k_i \in \N_0$, all $x_1,\dots,x_n\in [-n,n]^d$ and all $y_1,\dots,y_j \notin [-n,n]^d$.

\begin{proposition}\label{l1} Let $\{g(n)\}_{n\ge 0}$ be as in \eqref{crecientes}, and consider translation invariant transition probabilities $\{p(x,y)\}_{x,y \in \Z^d}$.  Let $\mu$ be a translation invariant probability measure on $X$ such that  $\int \eta(0)d\mu(\eta)<\infty$. Then, 
for  all  $t\geq 0$, the sequence $[\mu]_n S(t)$ converges as $n\to \infty$ to a probability measure $\mu_t$ on $X$ satisfying:
\begin{align} 
&\text{\it i) } \mu_t  \text{ is translation invariant,}
\notag\\
&\text{\it ii) } \int \eta(0)d\mu_t(\eta)\leq \int \eta(0)d\mu(\eta),
\notag\\
&\text{\it iii) } \text{ Semigroup property: for } s,t\geq 0,\, \mu_{t+s}=(\mu_t)_s,
\notag\\
\intertext{and if $\mu_n$ is an increasing sequence of probability measures on $X_f$ that converge weakly to $\mu$, then }
&\text{\it iv) } \lim_n \mu_nS(t)=\mu_t. 
\notag 
\end{align}
\end{proposition}

It follows from this proposition that if $\mu$ is a translation invariant measure with finite mean, then 
the process started from almost any configuration with respect to $\mu$ will not suffer explosions.
Unfortunately, we cannot deduce from this that the same holds for any given unbounded deterministic initial condition.

A natural question is whether equality holds in part ii) of this proposition. In \S \ref{Mass con} we  answer the question affirmatively when $p(x,y)$ corresponds
to a nearest neighbour random walk on $\Z$.

Given a parameter $\phi>0$, consider the  product measures $\mu_\phi$ with i.i.d.~marginal distributions 
\begin{align}
\label{invariant}
\mu_\phi\big(\eta(x)=k\big)=\frac{1}{z(\phi)} w(k) \phi^k, \qquad x\in \Z^d, \, k\in \N_0, \\
\intertext{where}
w(0)=1 \qquad \text{and}\qquad w(k)=\prod_{j=1}^k \frac{1}{g(j)},\, k\ge 1.\notag
\end{align}
The parameter $\phi$ is called the fugacity, and the measure exists as long as 
\begin{equation*}
z(\phi):=\sum_{k=0}^\infty w(k) \phi^k <\infty.
\end{equation*}
With the hypotheses \eqref{crecientes}  the measures are well defined for all choices of $\phi$, and they have finite moments of all orders. The particle density is given by 
\begin{equation*}
R(\phi):=E^{\mu_\phi}[\eta(x)]=\frac{1}{z(\phi)} \sum_{k\ge 1} k w(k) \phi^k
\end{equation*}
which turns out to be strictly increasing in the parameter $\phi$,
\[
\partial_\phi R(\phi)=\frac{1}{\phi} \Big\{\frac{1}{z(\phi)}\sum_{k=1}^\infty  k^2 w(k) \phi^k\,-\,\frac{1}{z^2(\phi)} \big(\sum_{k=1}^\infty k w(k) \phi^k\big)^2 \Big\}>0
\]
by Jensen's inequality. We also point out that $\lim_{\phi\to 0}R(\phi)=0$ and
\begin{equation*}
R(\phi)\nearrow \infty \quad \text{ as $\phi\to \infty$.}
\end{equation*}
Finally, for any $\phi>0$ we have
\begin{equation}
\label{fugacity}
E^{\mu_\phi}[g(\eta(x))]=\frac{1}{z(\phi)}\sum_{k=1}^\infty g(k) \phi^k\,w(k)=\phi,\quad x\in \Z^d.
\end{equation}

It is known that the measures $\{\mu_\phi\}_{\phi>0}$ are invariant for the zero-range dynamics when this is well defined, see e.g. \cite{A, S}. 
The following result states that this remains so under our hypotheses.

\begin{theorem}\label{l2} \textbf{Invariant measures}\\
Let the rates $\{g(n)\}_{n\ge 0}$ be as in \eqref{crecientes}. Consider translation invariant transition probabilities $\{p(x,y)\}_{x,y \in \Z^d}$. Then  $\mu_{\phi}$ satisfies $(\mu_\phi)_t=\mu_\phi$, for all $\phi\in (0,\infty)$ and $t>0$.
\end{theorem}
 
 For $\eta \in X$ and $n\in \N$ let
\begin{equation}
\label{eta-n}
\eta^n(x)=\eta(x) \quad\text{ if }\quad |x|\leq n \quad \text{ and }\quad \eta^n(x)=0\quad  \text{ if }\quad |x| > n.
\end{equation}
The process $\eta^n_t$ with initial value $\eta^n$ is
 well defined.
 
Let us now recall the following graphical construction, first developed by Harris in \cite{Harris}, which for simplicity we describe in the particular setting of translation invariant zero-range processes.

{\it Graphical representation.} Let $\{p(z)\}_{z\in \Z^d}$ be such that $p(z)\ge 0,\,z\in \Z^d$, $\sum_z p(z)=1$. Independently for each site $x,\  x\in \Z^d$, consider a marked Poisson process 
\begin{align}
\label{pp1}
\Gamma^x_*=\big\{\big((y,t),m^x_{(y,t)}\big),\,(y,t)\in \Gamma^x\big\},
\end{align}
where $\Gamma^x$ is an intensity $1$-Poisson process on $\{(y,t)\in [0,\infty)\times[0,\infty)\}$, and the marks $\{m_{(y,t)},\,(y,t)\in \Gamma^x\}$ are i.i.d.~random vectors in $\Z^d$ with distribution $\{p(z)\}_{z\in \Z^d}$.
We will now give an explicit construction of the zero-process $(\eta_t,\,t\ge 0)$ associated to a finite initial configuration 
$\eta \in X_f$, a rate function $g(n)$ and an underlying translation invariant$\{p(y-x)\}_{x,y\in \Z^d}$ random walk.
Assume that the zero-range process $(\eta_s,\ s<t)$ has been built up to time $t-$. Then, if the Poisson point process $\Gamma^x$ has an atom at $(y,t)$ and $\eta_{t-}(x)\ge 1$,  the site $x$ will lose a particle if
$g\big(\eta_{t-}(x)\big)\ge y$. If that is the case, the particle will jump to the site $z\in \Z^d$ such that $m^x_{(y,t)}=z-x$. 

The  advantage of this method is that it allows us  to construct  zero-range processes for all initial finite-particle configurations in the same probability space. Furthermore, in the particular case $d=1$ we might improve the construction so that we can simultaneously construct all nearest neighbour zero-range processes. To see this, independently for each $x\in \Z$, consider a marked Poisson point process 
$\tilde{\Gamma}^x_*=\big\{\big((y,t),U^x_{(y,t)}\big),\,(y,t)\in \Gamma^x\big\},$ where as before $\Gamma^x$ is an intensity $1$-Poisson process on $\{(y,t)\in [0,\infty)\times[0,\infty)\}$, and the marks are independent uniform random variables,  $U^x_{(y,t)}\sim U[0,1],\,(y,t)\in \Gamma^x$.
For given transition probabilities $p(1)=p,\, p(-1)=q=1-p$ and $x\in \Z$,  let $m^{x,(p,q)}_{(y,t)}=1$ if $U^x_{(y,t)}\le p$, and $m^{x,(p,q)}_{(y,t)}=-1$ otherwise. Then the process $(\eta^{(p,q)}_t,\,t\ge 0)$ constructed using the atoms and marks of the marked Poisson process
$\tilde{\Gamma}^{x,(p,q)}_*:=\big\{\big((y,t),m^{x,(p,q)}_{(y,t)}\big),\,(y,t)\in \Gamma^x\big\},$ is a nearest neighbour zero-range process with underlying $(p,q)$-dynamics.
  
Applying the graphical representation to simultaneously construct the processes $(\eta_t^n,\,t\ge 0)$, we obtain $\eta^n_t(x)\leq \eta^{n+1}_t(x)$ for all $x,\,t,\,n$.
 We can now let 
 \begin{equation}
 \label{climit}
 \eta_t(x) =\lim_n \eta^n_t(x),
 \end{equation}  
where the process $\eta_t$ takes values in $({\N_0\cup \{\infty\}})^{\Z^d}$.
The rest of the article focuses on finding a subset $Y\subseteq X$ and conditions on the rates and jump distributions so that $(\eta_t,\,t\ge 0)$ is a Markov process on $Y$. 

In $\S$ 
\ref{sec:1d} we consider the one-dimensional case $X=\N_0^\Z$ with nearest neighbour transitions, $p(x,y)=0$ if $|x-y|>1$. Then, we let  \begin{equation}
 \label{Yset}
 Y=\{\eta \in X: \limsup \frac{1}{2n+1}\sum_{x=-n}^n \eta(x)<\infty  \}
 \end{equation}
 be the space of configurations with bounded Ces\`aro mean,
and prove: 
\begin{theorem}\label{Markov1}
Let  $d=1$, let $\{g(n)\}_{n\ge 0}$ be as in \eqref{crecientes}, and let $\{p(x,y)\}_{x,y \in \Z}$ be the transition probability matrix of a nearest neighbour random walk. If $\eta_0 \in Y$, then $(\eta_t,\, t\ge 0)$ is a Markov process on $Y$. 
\end{theorem}

Next, we show that this process solves the martingale problem associated to the generator \eqref{generator}. Note that when $p(\cdot, \cdot)$ is symmetric our proof requires that the rate function $g$ is bounded by an exponential function.

\begin{theorem}\label{martingales1}
Let $d=1$, let $\{g(n)\}_{n\ge 0}$ be as in \eqref{crecientes} and let $\{p(x,y)\}_{x,y \in \Z}$ be the transition probability matrix of a nearest neighbour random  walk.
\newline i) If $p(0,1)-p(0,-1)\neq 0$ then  $(\eta_t,t\geq 0)$ is a solution of the martingale problem associated to $L$ and \eqref{integrated} holds.
\newline ii) If $g$ is bounded by an exponential function, $g(n)\le c e^{\theta n}$, for some $c,\, \theta> 0$, and all $n\in \N$, then $(\eta_t,t\geq 0)$ is a solution of the martingale problem associated to $L$ and \eqref{forward} holds.
\end{theorem}
 
In $\S$\ref{sec:alter} we find an alternative set of conditions ensuring the good definition of the process; these are more restrictive on the jump rates, but allow for general finite range transitions and any dimension.
In order to derive the results of that section we need to construct our process with a different limiting procedure. Given $\eta \in X$ we enumerate the particles of $\eta$ in an arbitrary manner. Then we let $x^i$ be the position of the i-th particle
and for each $N\in \N$ we define
\begin{equation*}
\eta^N(z) := \sum_{i=1}^N \mathbf{1}(x^i = z), \qquad z\in \Z^d
\end{equation*}
and
\begin{equation}
\label{h}
h(n) :=\sup_{1\leq j\leq n}(g(j)-g(j-1)).
\end{equation}
Then, for a given continuous time random walk on $\Z^d$ and $z\in \Z^d$ we let $\tau_0$ be the hitting time of the origin, $P^z$ the law of the random walk starting from $z$ and
\begin{equation*}
F_z(t) :=P^z(\tau_0\leq t).
\end{equation*}
Finally, for $\eta \in X$, $t\geq 0$ and $z\in \Z^d$ we let
\begin{equation*}
 \overline m_z(t,\eta)=\sum_{i\in \N}F_{x^i-z}(h(i)t).
\end{equation*}
In $\S$\ref{sec:alter} we will see that $\overline m_z(t,\eta)$ is an upper bound for the expected number of particles reaching $z\in \Z^d$ over the time interval $[0,t]$, for the initial configuration $\eta$. 

 The following result is a version of Theorem \ref{martingales1} with different hypotheses on the jump rates and transition probabilities, for general dimension $d\ge 1$.
\begin{theorem}
\label{t1.milton}
Let $\eta_0 \in X$, $g: \bb N_
0 \to [0,\infty)$ as in \eqref{crecientes}, $T >0$, $p(x,y)$  a set of finite-range, translation-invariant transition probabilities such that
\begin{itemize}
\item[i)]  for any $z \in \Z^d$ 
\begin{equation}
\label{As1}
\overline m_z(T,\eta_0)<\infty,
\end{equation}
\item[ii)] there exist positive 
$\theta, c$ such that $g(n) \leq c e^{\theta n}$ for any $n \in \bb N_0$.
\end{itemize}
Then the process $\{M_t^f; t \in [0,T]\}$ given by
\begin{equation}
\label{forward2}
M_t^f:= f(\eta_t) -f(\eta_0) - \int_0^t L f(\eta_s) ds
\end{equation}
for any $t \in [0,T]$,  is a martingale for any local, bounded function $f:X \to \bb R$. Moreover
\eqref{forward} holds.
\end{theorem}

An application of this theorem is given by the following corollary which requires a hypothesis on the growth of $h$ rather than $g$.
\begin{corollary}\label{finite_m}
Assume that the initial configuration of particles has finite Ces\`aro mean  $\rho=\limsup_{m\to \infty} \frac{1}{(2m+1)^d}\sum_{\|x\|\le m} \eta_0(x)$ and that the rates satisfy 
\eqref{crecientes}. Let  $p(x,y)$ be translation invariant, finite-range transition probabilities and let $h$ be as in \eqref{h}. Then, the conclusions of Theorem \ref{t1.milton} hold 
for any $T>0$, if either of the following conditions is satisfied,
\begin{itemize}
\item[a)]
$\sum_{z \in \bb Z^d} z\,p(0,z) =0$
and there exists  $a < \frac{2}{d}$ such that $\sup_n h(n)\,n^{-a}<\infty$.
\item[b)] $\lim_n h(n)\,n^{-1/d}=0 $.
\end{itemize}

\end{corollary}

Finally, we note that by the law of 
large numbers, for any $\phi>0$, the measure $\mu_\phi$ in \eqref{invariant} is supported on $Y$. Then,
once the zero-range process is well defined, for instance under the hypotheses of Theorems \ref{Markov1} or \ref{t1.milton}, Theorem \ref{l2} identifies a family of translation 
invariant, invariant measures.

The proofs in this paper are presented as follows: In $\S$ \ref{sec:c+im}  we prove Propostion \ref{l1}, Theorem \ref{l2} and a lemma establishing some properties of the invariant measures $\mu_{\phi}$.  
In $\S$ \ref{sec:1d} we restrict the setting to $\Z$ with nearest neighbour transition matrices $p(x,y)=0$ if $|x-y|>1$. We first show that the process started from an arbitrary configuration in $Y$ does not undergo explosions,  then that it satisfies the Markov property and after that a conservation of mass property for translation invariant initial distributions. Finally, in the last part of that section we prove Theorem \ref{martingales1}. In $\S$ \ref{sec:alter} we prove Theorem \ref{t1.milton} and Corollary \ref{finite_m}. We conclude the paper stating some open problems in $\S$ \ref{op}.

\section{Translation invariant initial distributions\label{sec:c+im}}

In this section we first prove Proposition \ref{l1}, where the zero-range process is constructed for translation invariant initial distributions having a finite mean.
We then prove Theorem \ref{l1} concerning invariant measures. 

\begin{proof}{Proposition}{l1}

To prove ii)
let $\{[\mu]_n\}_{n\in \N}$ be the family of probability measures associated to $\mu$  as in \eqref{mu_n}. We now show that
\begin {equation*}
\int \eta(x)d[\mu]_n S(t)\leq \int \eta(0)d\mu(\eta)
\end{equation*}
for all $t\geq 0$, $x\in \Z^d$ and $n\in \N$.
Let $c_{t,n}(x)=\int \eta(x)d[\mu]_n S(t)$. This is increasing in $n$ and therefore converges to a limit $c_t(x)\in [0,\infty]$.  Since $T^{y-x}[\mu]_n\leq [\mu]_{n+\|y-x\|_{\infty}}$,
 for all $x,y\in \Z^d$, all $n\in \N$ and $t\geq 0$
we have
\begin{equation}
\label{e2}
\begin{split}
c_{t,n}(x)&=\int \eta(x)d[\mu]_nS(t)= \int(T^{y-x}\eta)(y)d[\mu]_nS(t) \\
 &=\int \eta(y)dT^{y-x}[\mu]_nS(t)\leq \int \eta(y)d[\mu]_{n+\|y-x\|}S(t)=c_{t,n+\|y-x\|_{\infty}}(y).
\end{split}
 \end{equation}

Taking limits as $n$ goes to infinity we get $c_t(x)\leq c_t(y)$ and exchanging the roles of $x$ and $y$ we get the opposite inequality. Hence $c_t(x) $ 
does not depend on $x$ and we rename it $c_t$.
Fix $\epsilon >0$. Then there exists $n_0$ such that for all $n\geq n_0$
\[
\int \eta(0)d[\mu]_n S(t)\geq M- \epsilon,\quad \text{for any constant }M<c_t.
\]
It then follows from \eqref{e2} that:
\begin{equation}
\label{e3}
\sum_{y: \|y\|_{\infty}\leq m} \int \eta(y)d[\mu]_{n_0+m}S(t)\geq (2m+1)^d (M-\epsilon)\,.
\end{equation}
But the left hand side of \eqref{e3} is bounded above by
\begin{equation}
\label{e4}
\begin{split}
\sum_{y\in \Z^d} \int \eta(y)d[\mu]_{n_0+m}S(t)&= \sum_{y\in \Z^d} \int \eta(y)d[\mu]_{n_0+m} \\
&=(2n_0+2m+1)^d \int \eta(0)d\mu\,,
\end{split}
\end{equation}
where the first equality follows from the fact that the number of particles of a finite initial configuration is conserved.
It now follows from  \eqref{e3} and \eqref{e4} that
\begin{equation*}
M-\epsilon \leq \frac {(2n_0+2m+1)^d}{(2m+1)^d}\int \eta(0)d\mu ,
\end{equation*}
and letting $m$ go to infinity we get
\[
M-\epsilon \leq \int \eta(0)d\mu\,. 
\]
Since $\epsilon$ is arbitrary we conclude that
\[M\leq \int \eta(0)d\mu\ .\]
As  $M$  is any number strictly smaller than $c_t$, this implies that $c_t\leq \int \eta(0)d\mu<\infty$. Using again the assumption
that $\int \eta(0)d\mu < \infty$, we see that  the sequence 
$\{[\mu]_nS(t)\}_{n\in \N}$ is tight, and from the fact that it
is increasing it follows that it must converge to a measure $\mu_t$  with   $\int \eta(0)d\mu_t\leq  \int \eta(0)d\mu$ and ii) is
proved.

Since $[\mu]_n\leq T^x[\mu]_{n+\|x\|_{\infty}}$ and the semigroup preserves the order of measures we have: 
\[[\mu]_n S(t) \leq T^x[\mu]_{n+\|x\|_{\infty}} S(t).\]
Taking limits as $n$ goes to infinity we get
$\mu_t\leq T^x\mu_t$. Since the opposite inequality can be proved in the same way, i) follows.

Let now $\{\mu_n\}_{n\in \N}$ be an increasing sequence of probability measures on $X_f$ converging weakly to $\mu$. For all $k\ge 1$ we have $[\mu_n]_k \leq \mu_n$,  hence
\[
[\mu_n]_kS(t)\leq \mu_n S(t),
\]
and therefore
\[
 [\mu]_kS(t)\leq \lim_n \mu_n S(t).
\]
Taking limits in $k$ we get
\[
 \mu_t \leq \lim_n \mu_n S(t).
\]
Then, using the fact that as $k$ goes to infinity, $[\mu_n]_k$ increases to $\mu_n$, write
\[
\mu_n S(t)=\lim_k[\mu_n]_k S(t)\leq \lim_k [\mu]_k S(t)=\mu_t,
\]
thus proving iv).

In order to obtain iii), note that 
\[
\mu_{t+s}= \lim_n [\mu]_n S(t+s)=\lim_n ([\mu]_n S(t)) S(s).
\]
Since $[\mu]_n S(t)$ increases to $\mu_t$, the result follows from iv).
\end{proof}

We now turn to the proof of Theorem \ref{l2}. Recall the definition \eqref{invariant} of the family of translation invariant, product measures $\{\mu_\phi\}_{\phi>0}$. Given a measure $\mu$ on $X$ we will consider its projection  on $X_n:=\N_0^{[-n,n]^d}$,
\begin{equation}
\label{pi_n(mu)}
\Pi_n(\mu)(\xi)=\mu\big(\eta \in X,\,\ \eta(x)=\xi(x) \, \forall\,x\in [-n,n]^d\big),\quad \xi \in X_n.
\end{equation}
Note that while $[\mu]_n$  as in \eqref{mu_n} is a probability measure on $X$, $\Pi_n(\mu)$ is a probability measure on $X_n$. Clearly
$\Pi_n(\mu)=\Pi_n([\mu]_m)$ for any $m\geq n$.

\begin{proof}{Theorem}{l2}
For $n\in \N$ we  say that $x,y\in \Z^d$ are equivalent ($x\equiv y$) if all the coordinates of $x-y$ are multiples of $2n+1$. 
Define a transition matrix $p_n$ on $[-n,n]^d$ as follows:
for $x,y \in [-n,n]^d$ let 
\[p_n(x,y)=\sum_{z:\, z\equiv y}p(x,z).\]
A standard computation shows that the measures $\Pi_n(\mu_{\phi})$  are invariant for the periodic zero-range process on
$X_n$ with transition probability matrix $p_n(\cdot ,\cdot )$. Call $\hat S_n(t)$ the semigroup
associated to this process. Now define a new process on $X_n$. In this new process particles jump as in the original process 
following the transition matrix $p(\cdot,\cdot )$ but when a particle jumps to a point off $[-n,n]^d$ it vanishes. Call 
$\bar S_n(t)$ its semigroup. 
By means of appropriate couplings we will now prove that
\begin{equation} \label{bc}
\Pi_n(\mu_{\phi})\bar S_n(t)\leq \Pi_n\big([\mu_{\phi}]_nS(t)\big) 
\end{equation}
and
\begin{equation}\label{wc}
\Pi_n(\mu_{\phi})\bar S_n(t)\leq \Pi_n(\mu_{\phi})\hat S_n(t)=\Pi_n(\mu_{\phi}). 
\end{equation}
 Note that the measures involved in these expressions are supported on $X_f$. 
 To prove \eqref{bc} first let 
 \begin{equation}
\label{eta-x}
\eta^{x}(z)=\begin{cases}\eta(x)-1, &z=x \text{ and } \eta(x)\ge 1, \\ \eta(z) & \text{ otherwise.} \end{cases}
\end{equation} 

Then,  consider 
 the following generator of a Markov Process $(\eta_t,\xi_t)$ on $X_n\times X_f$:
\begin{equation}
\begin{split}
\hspace{-9mm}{\cal G}f(\eta,\xi)&=\sum_{x,y\in [-n,n]^d,\,x\neq y} g(\eta(x) \wedge \xi(x)) \,p(x,y)\big(f(\eta^{x,y}, \xi^{x,y})\big)-f(\eta,\xi)\big) \\
&\quad + \sum_{x,y\in [-n,n]^d,\,x\neq y} [g(\eta(x))-g(\eta(x) \wedge \xi(x))]\,p(x,y) \big(f(\eta^{x,y}, \xi)\big)-f(\eta,\xi)\big) \\
&\quad + \sum_{x,y\in [-n,n]^d,\,x\neq y} [g(\xi(x))- g(\eta(x) \wedge \xi(x))] \,p(x,y) \big(f(\eta, \xi^{x,y})\big)-f(\eta,\xi)\big) \\
&\quad +\sum_{x\in [-n,n]^d,y\in \Z^d\setminus [-n,n]^d} [g(\eta(x) \wedge \xi(x))]\,p(x,y) \big(f(\eta^{x}, \xi^{x,y})\big)-f(\eta,\xi)\big) \\
&\quad +\sum_{x\in [-n,n]^d,y\in \Z^d\setminus [-n,n]^d}[g(\eta(x))-g(\eta(x) \wedge \xi(x))]\,p(x,y) \big(f(\eta^{x}, \xi)\big)-f(\eta,\xi)\big) \\
&\quad + \sum_{x \in  [-n,n]^d,\, y\in \Z^d\setminus [-n,n]^d} [g(\xi(x))- g(\eta(x) \wedge \xi(x))] \,p(x,y) \big(f(\eta, \xi^{x,y})\big)-f(\eta,\xi)\big) \\
&\quad + \sum_{x \in  \Z^d\setminus [-n,n]^d} g(\xi(x))\big(f(\eta, \xi^{x,y})\big)-f(\eta,\xi)\big) .\\
\end{split}
\end{equation}  
In the same fashion as basic coupling \eqref{basic}, this generator matches the evolution of $(\eta_t,\,t\ge 0)$ and $(\xi_t,\,t\ge 0)$as much as possible, and supplements the rates so that the semigroups 
of the first and second marginals of the process with generator $\cal G$ are $\bar S_n(t)$ and $S(t)$ respectively. Moreover, if this process starts from a configuration $(\eta,\xi)$ such that $\eta(x)\leq \xi(x)$ for all $x\in [-n,n]^d$, then
$P(\eta_t(x)\leq \xi_t(x))=1$ for all $ x\in [-n,n]^d$ and all $ t\geq 0$. Hence, \eqref{bc} follows.
 
 To prove \eqref{wc} we procede similarly: we consider the following generator of a Markov Process on $X_n \times X_n$,
 \begin{equation}
\begin{split}
\hspace{-9mm}{\cal H}f(\eta,\xi)&=\sum_{x,y\in [-n,n]^d,\,x\neq y} g(\eta(x) \wedge \xi(x)) \,p(x,y)\big(f(\eta^{x,y}, \xi^{x,y})\big)-f(\eta,\xi)\big) \\
&\quad + \sum_{x,y\in [-n,n]^d,\,x\neq y} [g(\eta(x))-g(\eta(x) \wedge \xi(x))]\,p(x,y) \big(f(\eta^{x,y}, \xi)\big)-f(\eta,\xi)\big) \\
&\quad + \sum_{x,y\in [-n,n]^d,\,x\neq y} [g(\xi(x))- g(\eta(x) \wedge \xi(x))] \,p(x,y) \big(f(\eta, \xi^{x,y})\big)-f(\eta,\xi)\big) \\
&\quad +\sum_{x,y \in [-n,n]^d}\big( \sum_{z\neq y, z\equiv y}p(x,z)\big) [g(\eta(x) \wedge \xi(x))]\,p(x,y) \big(f(\eta^{x}, \xi^{x,y})\big)-f(\eta,\xi)\big) \\
&\quad +\sum_{x,y \in [-n,n]^d}\big( \sum_{z\neq y, z\equiv y}p(x,z)\big) [g(\eta(x))-g(\eta(x) \wedge \xi(x))] \big(f(\eta^{x}, \xi)\big)-f(\eta,\xi)\big) \\
&\quad +\sum_{x,y \in [-n,n]^d}\big( \sum_{z\neq y, z\equiv y}p(x,z)\big)  [g(\xi(x))- g(\eta(x) \wedge \xi(x))] \big(f(\eta, \xi^{x,y})\big)-f(\eta,\xi)\big) .\\
\end{split}
\end{equation}  
 The semigroups of the first and second marginals of the process with generator $\cal H$ are $\bar S_n(t)$ and $\hat S_n(t)$ respectively. And if $\eta(x)\leq \xi(x)$ for all $x\in [-n,n]^d$, then
$P(\eta_t(x)\leq \xi_t(x))=1$ for all $ x\in [-n,n]^d$ and all $ t\geq 0$. Hence, the inequality in \eqref{wc} follows. Finally, the equality in \eqref{wc} is due to the invariance under $\hat S_n(t)$ of $\Pi_n(\mu_{\phi})$.

 Now, fix $\phi>0$
and let  
\[
\rho=\int \eta(x)d\mu_{\phi} (\eta),
\]
which does not depend on $x$.
For $n\geq \|x\|_{\infty}$ let 
\[ 
0\le d_{t,n}(x)= \rho-\int \eta(x)d\big(\Pi_n(\mu_{\phi}) \bar S_n(t)\big).
\]
Note that this function is  decreasing  in $n$.
Let
\begin{align}
\label{2-1}
d_t(x)=\lim_n d_{t,n}(x).
\end{align}
To prove that \eqref{2-1} does not depend on $x$ we apply analogous arguments to those used to derive \eqref{e2}. By coupling and the fact that $n-\|x\|_\infty \le (n+\|x-y\|_\infty) -\|y\|_\infty$ we see that
\[
\int \eta(y)d\big(\Pi_{n+{\|x-y\|}_{\infty}}(\mu_{\phi}) \bar S_{n+\|x-y\|_{\infty}}(t)\big)\geq   \int \eta(x)d\big(\Pi_n(\mu_{\phi}) \bar S_n(t)\big)
\]
and 
\[
d_{t,n+{\|x-y\|}_{\infty}}(y)\leq d_{t,n}(x)\,.
\]
Hence $d_t(y)\leq d_t(x)$ and role reversing  $x$ and $ y$ we conclude
that $d_t(x)=d_t(y)$ for all $x,y\in \Z^d$.

To derive an upper bound for
$\sum_{x\in [-n,n]^d}d_{t,n}(x)$, write
\begin{equation}
\label{dtn}
\sum_{x\in [-n,n]^d}d_{t,n}(x)=
\sum_{x\in [-n,n]^d}\Big(\int \eta(x) d\big(\Pi_n(\mu_{\phi})\hat S_n(t)\big)-\int \eta(x) d\big(\Pi_n(\mu_{\phi})\bar S_n(t)\big)\Big).
\end{equation}
Coupling the processes with semigroups $\hat S_n$ and $\bar S_n$ and starting them  from the same random initial configuration distributed  according to
$\Pi_n(\mu_{\phi})$, we see that the rate at which
\eqref{dtn} increases with $t$ 
is, at all times, bounded above by
\[
A_n=\sum_{\substack{x\in [-n,n]^d,\\ y\notin [-n,n]^d}}p(x,y)\int \eta(x)d\big(\Pi_n(\mu_{\phi})\big) =\rho \sum_{\substack{x\in [-n,n]^d,\\ y\notin [-n,n]^d}}p(x,y).
\]
Hence 
\[
\sum_{x\in [-n,n]^d}\Big(\int \eta(x) d(\mu_{\phi})\hat S_n(t)-\int \eta(x) d(\mu_{\phi})\bar S_n(t)\Big)\leq tA_n
\]
and
\[
0\leq \lim_n \frac{1}{n^d}\sum_{x\in [-n,n]^d}d_{t,n}(x)\leq \frac{tA_n}{n^d}=0. 
\]

But since $d_{t,n}(x)$ is decreasing in $n$ and its limit does not depend on $x$, this can only happen if $d_t(x)=0$.
Together with \eqref{wc} this
implies that 
the finite-dimensional distributions of 
 $\Pi_n(\mu_{\phi})\bar S_n(t)$
 increase as $n \to \infty$  to the finite-dimensional distributions of $\mu_{\phi}$.
It now follows from \eqref{bc} and part ii) of Proposition \ref{l1} that  $(\mu_{\phi})_t=\lim_n   [\mu_{\phi}]_n S(t)=\mu_{\phi}$.
\end{proof}

We finish this section with a lemma describing some simple properties of the invariant measures $\mu _{\phi}$.

\begin{lemma}\label{inv-m}
Assume $\lim_n g(n)=\infty$ and that $\eta \in X$ is distributed according to $\mu _{\phi}$, then
\begin{equation*} 
\sum_{k\in \N} \exp(\gamma k)\mu_{\phi}(\eta(x)=k)<\infty\quad  \forall \phi,\gamma>0,
\end{equation*}
and for any $\alpha>0$
\begin{equation}
\label{ld}
\lim_{\phi \to \infty }\limsup_{k\to \infty}\frac{1}{k} \log\mu_{\phi}\Big( \sum_{x=1}^k \eta(-x) \le \alpha k\Big)=-\infty.
\end{equation}
\end{lemma}
\begin{proof}{Lemma}{inv-m}
The first statement follows immediately from the divergence of $g(k)$.  For the second statement, first note that for any $M>0$
\[
\lim_{\phi \to \infty} \mu_{\phi }(\eta(x)\leq M )=0. 
\]
Hence,
\[
\lim_{\phi \to \infty} \mu_{\phi }\left(\frac{\eta(x)}{\alpha+1}\leq 1\right)=0.
\]
Therefore for any $0<p<1$ there exists a $\phi(p)$ such that for any $\phi\geq \phi(p)$ we have
 \[
  \mu_{\phi}\left( \sum_{x=1}^k \eta(-x) \le \alpha k\right) \leq P\left(\sum_{i=1}^k X_i\leq \frac {\alpha}{\alpha+1}k\right),
 \]
 where $\{X_i\}_{i\ge 1}$ are i.i.d.~Bernoulli with parameter $p$.  But the right hand side above is equal to
 \[
  P\left(\sum_{i=1}^k Y_i\geq \frac {1}{\alpha+1}k\right)
 \]
 where $Y_1,Y_2,\dots,Y_k$ are i.i.d.~Bernoulli with parameter $1-p$. Using the expression for the large deviation rate of Bernoulli random variables, see for instance \cite{MR978907},
 we see that for any $K>0$
 \[
  \limsup_{k\to \infty}\frac{1}{k} \log P\left(\sum_{i=1}^k Y_i\geq \frac {1}{\alpha+1}k\right)\leq -K
 \]
 if $1-p$ is small enough. Hence \eqref{ld} holds.
\end{proof}

\section{Dimension 1, nearest neighbour transitions\label{sec:1d}}
 
Throughout this section we assume that  $d=1$ and that $\{p(x,y)\}_{x,y \in \Z}$ corresponds to a translation invariant, nearest neighbour random walk on $\Z$,
\begin{align}
\label{1dhypo}
p(x,y)=\begin{cases}p &y=x+1,  \\ q & y=x-1, \\ 0 & \text{ otherwise,} \end{cases}
\end{align}
where $0\le p,\,q\le 1$, $p+q=1$.

Let $X_f$ and $Y$ be as in \eqref{finite}  and \eqref{Yset} respectively.
Since we will be following the evolution of individual particles, it will be convenient to consider elements of $Y$ as increasing limits of elements of $X_f$. Hence, for $\eta\in Y$ and $n\in  \N$ we define
$\eta^n=\big(\eta^n(x)\big)_{x\in \Z}$ as in \eqref{eta-n}.

\begin{lemma}\label{l3} \textbf{Non-explosion}\\
 Let $d=1$, $\{p(x,y)\}_{x, y \in \Z}$ be as in \eqref{1dhypo} and consider rates $\{g(n)\}_{n\ge 0}$ that satisfy \eqref{crecientes}. Let $\eta_0 \in Y$.
Then $(\eta_t,\,t\ge 0)$ defined by \eqref{climit} satisfies $\eta_t(x)<\infty$ a.s. for all $t\geq 0$ and all $x\in \Z$.
Moreover, $\eta_t \in Y\ a.s.$ for all $t\geq 0$. 
\end{lemma} 
\begin{proof}{Lemma}{l3}
Let $j: X^2_f \rightarrow \R$ and $r:X^2_f \rightarrow \R$ be given by
\begin{align*}
&j(\zeta,\psi)=\Big[\sup_{n\leq 0,m\geq 0}\sum_{x=n}^m (\zeta(x)-\psi(x))\Big]^+ 
\\
\intertext{and}
&r(\zeta,\psi)=[\zeta(0)-\psi(0)]^+ .
\end{align*}
Note that for all $\zeta,\psi \in X_f$ we have $r(\zeta,\psi)\leq j(\zeta,\psi)$. 

Given initial configurations $\zeta_0,\,\psi_0$ consider  the coupled processes $\zeta_t,\, \psi_t$ on $X_f^2$.  
We  claim that  $j(\zeta_t,\psi_t)$ only increases 
when a $\psi$ particle jumps off $0$. To justify this last statement, first note that if a $\psi$ and a $\zeta$ particle
jump together the value of $j$ remains unchanged.
Then look at jumps of a $\zeta$ particle not accompanied by a $\psi$ particle occurring at some time $s$, and consider the 
following cases,
\begin{enumerate}
\item The $\zeta$ particle jumps from $k<0$ to $k-1$ . In this case for any $n\leq 0$ and any  $m\geq 0$ the expression 
$\sum_{x=n}^m (\zeta(x)-\psi(x))$ either remains unchanged or decreases by one unit.
\item The $\zeta$ particle jumps from $k<0$ to $k+1$. Since no $\psi$ particle jumped, it must be the case 
that just before the jump we had $\zeta_{s-}(k)- \psi_{s-}(k)>0$. The expression  $\sum_{x=n}^m (\zeta(x)-\psi(x))$
only increases if $n=k+1$. But
$\sum_{x=k+1}^m (\zeta_s(x)-\psi_s(x))=\sum_{x=k+1}^m (\zeta_{s-}(x)-\psi_{s-}(x))+1\leq \sum_{x=k}^m (\zeta_{s-}(x)-\psi_{s-}(x))$ 
and therefore $j$ does not increase.
\item The $\zeta$ particle jumps from $k>0$ to $k+1$ . 
In this case for any $n\leq 0$ and any  $m\geq 0$ the expression  $\sum_{x=n}^m (\zeta(x)-\psi(x))$ either remains unchanged or
decreases by one unit.
\item The $\zeta$ particle jumps from $k>0$ to $k-1$. Since no $\psi$ particle jumped, it must be the case that just before
the jump we had $\zeta_{s-}(k)- \psi_{s-}(k)>0$. The expression   $\sum_{x=n}^m (\zeta(x)-\psi(x))$
only increases if $m=k-1$. But
$\sum_{x=n}^{k-1} (\zeta_s(x)-\psi_s(x))=\sum_{x=n}^{k-1} (\zeta_{s-}(x)-\psi_{s-}(x))+1\leq \sum_{x=n}^k (\zeta_{s-}(x)-\psi_{s-}(x))$
and therefore $j$ does not increase.
\item The $\zeta$ particle jumps from $0$ to either $1$ or $-1$. In this case $j$ either remains unchanged or decreases by one unit.
 \end{enumerate}
Next look at jumps of a $\psi$ particle not accompanied by a $\zeta$ particle occurring at some time $s$, and consider the following cases,
\begin{enumerate}
\item The $\psi$ particle jumps from $k<0$ to $k+1$ . In this case for any $n\leq 0$ and any  $m\geq 0$ 
the expression  $\sum_{x=n}^m (\zeta(x)-\psi(x))$ either remains unchanged or decreases by one unit.
\item The $\psi$ particle jumps from $k<0$ to $k-1$. Since no $\zeta$ particle jumped, it must be the case that just before
the jump we had $\zeta_{s-}(k)- \psi_{s-}(k)<0$. The expression  $\sum_{x=n}^m (\zeta(x)-\psi(x))$
only increases if $n=k$. But 
$\sum_{x=k}^m (\zeta_s(x)-\psi_s(x))=\sum_{x=k}^m (\zeta_{s-}(x)-\psi_{s-}(x))+1\leq \sum_{x=k+1}^m (\zeta_{s-}(x)-\psi_{s-}(x))$ 
and therefore $j$ does not increase.
\item The $\psi$ particle jumps from $k>0$ to $k-1$ . In this case for any $n\leq 0$ and any  $m\geq 0$
the expression  $\sum_{x=n}^m (\zeta(x)-\psi(x))$ either remains unchanged or decreases by one unit.
\item The $\psi$ particle jumps from $k>0$ to $k+1$. Since no $\zeta$ particle jumped, it must be the case that just before
the jump we had $\zeta_{s-}(k)- \psi_{s-}(k)<0$. The expression   $\sum_{x=n}^m (\zeta(x)-\psi(x))$
only increases if $m=k$.
But, $\sum_{x=n}^{k} (\zeta_s(x)-\psi_s(x))=\sum_{x=n}^{k} (\zeta_{s-}(x)-\psi_{s-}(x))+1\leq \sum_{x=n}^{k-1} 
(\zeta_{s-}(x)-\psi_{s-}(x))$ and therefore $j$ does not increase.
 \end{enumerate}
The only remaining case is when a $\psi$ particle jumps off $0$. In this case $j$ either remains unchanged or increases by one unit.

Therefore, if we denote by $N_t(\psi) $ the number of $\psi$ particles that jumped off $0$ on $[0,t]$, we get
\begin{equation}\label{g1}
r(\zeta_t,\psi_t)\leq j(\zeta_t,\psi_t)\leq j(\zeta_0,\psi_0)+ N_t(\psi)
\end{equation}
and
\begin{equation}
\label{g2}
\zeta_t(0)\leq \psi_t(0)+r(\zeta_t,\psi_t)\leq \psi_t(0)+j(\zeta_0,\psi_0)+ N_t(\psi).
\end{equation}
Fix now $\eta \in Y$ and let 
 \[
 \gamma=\gamma(\eta):=\max \Big\{\limsup \frac{1}{n}\sum_{x=1}^n \eta(x),  \limsup \frac{1}{n}\sum_{x=1}^n \eta(-x)  \Big\}.
 \]
 Let $\phi$ be large enough so that
 $\lim \frac{1}{n}\sum_{x=1}^n \xi(x)>\gamma$ for $\mu_{\phi}$-almost all $\xi$.
 Consider zero-range processes $\eta^n_{\,\cdot}$ and $\xi^n_{\,\cdot}$ having initial configurations 
 $\eta^n_0=\eta^n$ and $\xi^n_0=\xi^n$, with $\xi$ distributed according to $\mu_\phi$. By  \eqref{g2} we get
 \begin{equation*}
 \eta^n_t(0)\leq \xi^n_t(0)+j(\eta^n_0,\xi^n_0)+ N_t(\xi^n)\,.
 \end{equation*}
Taking limits as $n$ goes to infinity, we see that
$\eta_t(0)<\infty$ a.s. will follow from
\begin{enumerate}
\item[{\it i)}] $j(\eta,\xi)<\infty$ a.s.,
\item[{\it ii)}]$\lim_n \xi_t^n(0)<\infty$ a.s., and 
\item[{\it iii)}] $\lim_n N_t(\xi^n) <\infty$ a.s..
\end{enumerate}

To prove {\it iii)} note that $ \EE\big[N_t(\xi^n)\big]= \int_0^t \EE^{\xi}\big[g(\xi^n_s(0)\big]ds$.
Then
\begin{equation}
\label{N_2}
\begin{split}
\int \EE\big[ N_t(\xi^n)\big] d\mu_{\phi}(\xi)&= \int \Big( \int_0^t \EE^{\xi^n}\big[g(\xi^n_s(0)\big]ds\Big ) d\mu_{\phi}(\xi),\\
&=\int_0^t \Big( \int \EE^{\xi^n}\big[g(\xi^n_s(0) \big] d\mu_{\phi}(\xi)   \Big ) ds 
\end{split}
\end{equation}
by Tonelli's Theorem. 
Since the process is monotone in $n$ and $g$ is increasing, the RHS in \eqref{N_2} is bounded above by
\[ 
\int_0^t \Big( \int \EE^{\xi}\big[g(\xi_s(0)\big] d\mu_{\phi}(\xi)   \Big ) ds = t\int g(\xi(0))d\mu_{\phi}(\xi) <\infty 
\]
from the invariance of $\mu_\phi$ and the fact that $E^{\mu_\phi}\big[g(\xi(0))\big]<\infty$ \eqref{fugacity}.
With our choice of $\phi$ we get $j(\eta,\xi)<\infty$ for $\mu_{\phi}$--almost all $\xi$. This proves {\it i)}. Using  the invariance of $\mu_{\phi}$, write 
\[
\int \lim_n \xi^n_t(0)d\mu_{\phi}(\xi)=\int \xi_t(0)d\mu_{\phi}=\int \xi(0)d\mu_{\phi}<\infty
\] 
from where  {\it ii)} follows.

To prove that $\eta_t \in Y$, $t>0$,  we apply the second inequality in \eqref{g1} to $\eta_0^n$ and $\xi_0^n$ and take limits in $n$
to obtain
\[
j(\eta_t, \xi_t) \leq j(\eta_0,\xi_0) +N_t(\xi),
\]
which implies the desired result.
\end{proof}

We now prove the Markov property.  The proof relies on the graphical representation described in \S~\ref{n&r}.
Let us denote by $w=\{\Gamma_*^x\}_{x\in \Z}$ the collection of independent, marked Poisson processes labelled by sites in $\Z$, see \eqref{pp1} . For $x\in \Z$ and $s<t$,  let $\Gamma^{x,(s,t]}_*$ be the Poisson atoms falling in $\R_{\ge 0}\times (s,t]\times \Z^d$ and $w_{s,t}=\big\{\Gamma^{x,(s,t]}_*\big\}_{x\in \Z}$. Then, given an initial configuration $\eta_0 \in X_f$, we can describe the state of the process $\eta_t$ at time $t$ as a function of 
its state $\eta_s$ at time $s$, and the updates $w_{s,t}$ occurring on $(s,t]$,
\begin{equation}
\label{update}
\eta_t=\Phi_{s,t}(\eta_s, w_{s,t})\quad a.s..
\end{equation}
The Markov property follows immediately from the independence of the distribution of the Poisson atoms in disjoint regions of $[0,\infty)\times[0,\infty)\times\Z^d$. Also, 
note that the mappings $\Phi_{s,t}(\cdot,\cdot)$ are monotone in the first coordinate, a consequence of the fact that the rates $\{g(n)\}_{n\ge 0}$ are non-decreasing.

We now show that it is possible to take limits in \eqref{update} to extend it to initial configurations in $Y$.

\begin{proof}{Theorem}{Markov1}
Given $\eta \in Y$, consider $\eta^n\nearrow \eta$, the sequence of processes $(\eta_t^n,\,t\ge 0)$ and the process $(\eta_t,\,t\ge 0)$ constructed in \eqref{eta-n}, so that in particular
$\eta_t^n\le \eta_t^{n+1}\le \eta_t$, for all  $n\in \N$ and $t\ge 0$.

In order to prove the theorem, it will be enough to show that the process $(\eta_t,\,t\ge 0)$ 
satisfies, for any $s<t$,
\begin{equation*}
\eta_t=\lim_{k\to \infty} \Phi_{s,t}\big((\eta_s)^k, w_{s,t}\big) \quad a.s..
\end{equation*}
From $\lim_{n\to \infty} \eta_s^n \ge (\eta_s)^k$, $k\in \N$, we have
\begin{equation*}
\eta_t=\lim_{n\to \infty} \eta_t^n=\lim_{n\to \infty} \Phi_{s,t}(\eta^n_s, w_{s,t})\ge \Phi_{s,t}\big((\eta_s)^k,w_{s,t}\big)\quad a.s..
\end{equation*}
On the other hand, for each fixed $n\in \N$, 
\[
\Phi_{s,t}\big(\eta_s^n, w_{s,t}\big)=\lim_{k\to \infty} \Phi_{s,t}\big((\eta^n_s)^k, w_{s,t}\big)\le \lim_{k\to \infty}\Phi_{s,t}\big((\eta_s)^k, w_{s,t}\big)
\]
and taking limits in $n$ this yields the opposite inequality, 
\begin{equation*}
\eta_t= \lim_{n\to \infty} \Phi_{s,t}\big(\eta_s^n, w_{s,t}\big)\le \lim_{k\to \infty}\Phi_{s,t}\big((\eta_s)^k, w_{s,t}\big)\quad a.s..
\end{equation*}
The result follows.
\end{proof}

\subsection{Mass conservation \label{Mass con}}

The next lemma states that when the initial configuration is distributed according to a translation invariant measure, mass is preserved. 

\begin{lemma}\label{l4}
 Let $d=1$, $\{p(x,y)\}_{x, y \in \Z}$ as in \eqref{1dhypo} and consider rates $\{g(n)\}_{n\ge 0}$ that satisfy \eqref{crecientes}.
Let $\mu$ be a translation invariant measure such that $\rho:=\mu\big[\eta(0)\big] <\infty$. Then 
$\rho=E^{\mu}[\eta_t(0)]$  for all $t\ge 0$.
\end{lemma}
\begin{proof}{Lemma}{l4}
By part ii) of Proposition \ref{l1} $E^{\mu}[\eta_t(0)]\leq \rho$. Hence, it suffices to show the opposite inequality.
Let $\phi\ge 0$ be such that $E^{\mu}\big[\eta(0)\big]\le E^{\mu_\phi}\big[\eta(0)\big]$. Consider coupled families of processes $(\eta^n_t,\,t\ge 0)$, $(\eta_t,\,t\ge 0)$, $(\xi^n_t,\,t\ge 0)$, $(\xi_t,\ t\ge 0)$, such that
\begin{enumerate}
\item $\eta_0\sim \mu$, and $\eta^n_0$ is obtained from $\eta$ as in \eqref{eta-n},
\item $\xi_0 \sim \mu_\phi$, and $\xi^n_0$ is obtained from $\xi$ as in \eqref{eta-n},
\item  $(\eta^n_t,\,t\ge 0)$ and $(\xi^n_t,\,t\ge 0)$ follow the dynamics $\big(\bar{S}_n(t)\big)_{t\ge 0}$ defined in the proof of Theorem \ref{l2},
\item $(\eta_t,\,t\ge 0)$ and $(\xi_t,\ t\ge 0)$ follow the zero-range dynamics $\big(S_t\big)_{t\ge 0}$ determined by \eqref{generator},
\item $\eta_t^n\le \eta_t$ for all $t\ge 0$ a.s., and $\xi_t^n\le \xi_t$ for all $t\ge 0$ a.s.. 
\end{enumerate}
Processes with properties $1-5$ above can be constructed by means of the graphical representation, using the same family of marked Poison processes $\big\{\Gamma_*^x\big\}_{x\in \Z}$. For the construction of  $(\eta^n_t,\,t\ge 0)$ and $(\xi^n_t,\,t\ge 0)$, particles that jump out of $[-n,n]$ are removed. Let us denote by $P^{\mu\times\mu_{\phi}}$ the joint law of these four processes, such that $P^{\mu}$ is the marginal law of $(\eta_t,\, t\ge 0)$.

Denote by $\Psi_n(t)$ and $\Phi_n(t)$  the number of $\eta^n$ and $\xi^n$ particles that jump from $n$ to $n+1$  
(and are thereby lost) over the time interval $[0,t]$.
Let
\[
J_n(\eta,\xi)=[\sup_{x\in [-n,n]}\sum_{y\in [x,n]}(\eta(y)-\xi(y))]^+\,.
\]
Now note that $J_n(\eta^n_s,\xi^n_s)- \Phi_n(s)+\Psi_n(s)$ can only decrease in time.
Hence
\[
\Psi_n(t)\leq \Phi_n(t) -J_n(\eta^n_t,\xi^n_t)+ J_n(\eta^n_0,\xi^n_0)\leq \Phi_n(t)+ J_n(\eta^n_0,\xi^n_0),\quad \PP^{\mu \times \mu_{\phi}}- a.s..
\] 
Let
\[
G_n(\eta,\xi)=[\sup_{x\in (-\infty,n]}\sum_{y\in [x,n]}(\eta(y)-\xi(y))]^+.
\]
Then clearly
\[
\Psi_n(t)\leq \Phi_n(t)+ G_n(\eta^n_0,\xi^n_0) \quad\PP^{\mu \times \mu_{\phi}}-a.s.. 
\]
Now note that 
\[
\frac{\Phi_n(t)+  G_n(\eta_0,\xi_0)}{n}\longrightarrow 0 \quad \text{ in } \PP^{\mu \times \mu_{\phi}} \text{- probability as } n\to \infty,
\]
hence the same holds for $\Psi_n(t)/n$ in  $\PP^{\mu\times \mu_\phi}$-probability. Similarly, if we denote by $\Psi_{-n}(t)$ the number of $\eta$-particles jumping from $-n$ to $-(n+1)$ over
$[0,t]$, it follows that $\Psi_{-n}(t)/n$ converges to $0$ in $\PP^{\mu\times\mu_{\phi}}$-probability. We thus get
\[ 
\lim_n \frac{1}{n}\sum_{-n}^n [\eta_0(x)-\eta_t^n(x)]=0 \quad \text{ in } \PP^{\mu \times \mu_\phi}\text{- probability,}
\]
and in particular $ \lim_n (1/n)\sum_{-n}^n\eta_t^n(x)=\rho$ in $\PP^{\mu \times \mu_\phi}$-probability. Since $\eta_t(x)\geq \eta_t^n(x)$  $\PP^{\mu \times \mu_\phi}$~a.s.,
we get $ \lim_n (1/n)\sum_{-n}^n\eta_t(x)\geq \rho$ in $\PP^{\mu }$-probability.
But $\EE^{\mu}\big[\eta_t(x)\big]$ does not depend on $x$ by part i) of  Proposition \ref{l1}, hence the last inequality implies that
$\EE^{\mu}\big[\eta_t(0)\big]\geq \rho$. 
\end{proof}

\subsection{Proof of Theorem \ref{martingales1}\label{subsec:mart1}}

To prove the theorem we will need two lemmas and a proposition which we state here and prove later:
\begin{lemma}\label{asymmetric}\textbf{Asymmetric transitions}\\
Let $d=1$, $\{g(n)\}_{n\ge 0}$ as in \eqref{crecientes}, and consider $\{p(x,y)\}_{x, y \in \Z}$ as in \eqref{1dhypo} with $p\neq q$.
Then, for $\eta \in Y$, 
the distribution $P^\eta$ of the process $\big(\eta_t,\,t\ge 0: \eta_0=\eta\big)$, satisfies
\begin{equation}
\label{fg1}
E^\eta\Big[\int_0^t g\big(\eta_s(x)\big)\,ds\Big]<\infty
\end{equation}
for all $x \in \Z$ and $t>0$.
\end{lemma}

\begin{lemma}\label{finite-mean}\textbf{Exponentially bounded rates}\\
Let $d=1$, $\{g(n)\}_{n\ge 0}$ as in \eqref{crecientes}, and consider $\{p(x,y)\}_{x, y \in \Z}$ as in \eqref{1dhypo}.
Assume further that the rate function is exponentially bounded: there exists $\lambda>0$ such that
\[
g(n)\le e^{\lambda n}, \qquad n\in \N_0.
\]
Then for $\eta\in Y$, $Y$ the set in \eqref{Yset}, the distribution $P^{\eta}$ of the process $\big(\eta_t,\, t\ge 0 : \eta_0=\eta \big)$, satisfies 
\begin{equation}
\label{fg2}
\sup_{s\in [0,t]} E^{\eta}\big[g(\eta_s(x))^r\big]<\infty
\end{equation}
for any $r\in [1, \infty)$, $x\in \Z,$ and $t>0$.
\end{lemma}

\begin{proposition}
\label{forwardeqs}
Let $(\eta_t,\,t\ge0)$, $\eta_0\in Y$ be the process given by  \eqref{climit}. If, for some $r>1$, for all $t>0$, and any initial configuration
$\eta_0 \in Y,$  
\begin{equation}
\label{onemore}
\sup_{s\le t} E^{\eta_0}\big[g(\eta_s(x))^r\big] <\infty, \quad{\forall x\in \Z^d,}
\end{equation} 
then the process satisfies the forward equation:
\begin{equation*}
\frac{d}{dt}E^{\eta_0}[f(\eta_t)]=E^{\eta_0}\big[Lf(\eta_t)\big],
\end{equation*}
for any local bounded function $f:X\to \R$.
\end{proposition}

\begin{proof}{Theorem}{martingales1}
We start showing that $\{ \eta_t, t\geq 0\}$ is a solution of the martingale problem under the hypothesis of either item of the theorem. 
For each fixed $n\in \N$ and initial configuration $\eta_0 \in Y$, the process $\{\eta^n_t;\,t\ge 0\}$  is supported on the countable state space
\[
\big\{\eta \in X;\, \sum_x \eta(x)=\sum_{|x|\le n} \eta_0(x)\big\},
\] 
and 
the transition rates are bounded above by
$g\big(\sum_{|x|\le n} \eta_0(x)\big)<\infty$, so this is a Markov chain without explosions.  
Therefore, for any local, bounded function $f: X \to \bb R$ the process
\begin{align}
\label{martingale}
M_t^n(f) := f(\eta_t^n) - f(\eta_0^n) - \int_0^t L f(\eta_s^n) \,ds
\end{align}
is a martingale with quadratic variation 
\begin{align*}
\langle M_t^n(f)\rangle = \int_0^t \sum_{x,y \in \bb Z^d} g(\eta_s^n(x))p(x,y) \big(\nabla_{x,y} f(\eta_s^n)\big)^2 ds,
\end{align*}
where $\nabla_{x,y} f(\eta):=f(\eta^{x,y})-f(\eta)$ and $\eta^{x,y}$ is as in \eqref{eta-xy}.
Let $A$ be the support of $f$ and let $\bar{A}=\big\{y\in \Z,\, \inf_{x\in A}\|y-x\|\le1\big\}$. Then
\begin{equation}
\label{qvbound}
E\big[M_t^n(f)^2\big] = E\big[\langle M_t^n(f)\rangle\big] \leq 8 \|f\|_\infty^2\int_0^t \sum_{x \in \bar{A}} E[g(\eta_s^n(x)] ds.
\end{equation}

Due to \eqref{fg1} and \eqref{fg2}, in either item of the statements of the theorem, we can take the limit  as $n \to \infty$ in  \eqref{martingale}  and  conclude that the sequence $M_t^n(f)$ converges $a.s.$ and in $L^1$  to 
\begin{align*}
M_t(f) = f(\eta_t) - f(\eta_0) - \int_0^t L f(\eta_s)\,ds. 
\end{align*}
Now,  using Fatou's Lemma and  \eqref{qvbound}, we can control the second moment of the increment $M_t(f)-M_s(f)$:
\begin{align*}
E\big[(M_t(f)-M_s(f))^2\big]&=E\Big[ \liminf_{n\to \infty} \big(M^n_t(f)-M^n_s(f)\big)^2 \Big]\\ 
&\le \liminf_{n\to \infty} E\Big[ \big(M^n_t(f)-M^n_s(f)\big)^2 \Big]\\
&\le 8 \|f\|_\infty^2\int_s^t \sum_{x \in \bar{A}} E[g(\eta_s(x)] ds.
\end{align*}
 It then follows from the main result of \cite{Loynes} that $M_t(f)$ is a martingale. Now, for item i) of the theorem, \eqref{integrated} follows from Fubini's Theorem and Lemma \ref{asymmetric}, and for item ii) of the theorem, \eqref{forward} follows from Lemma \ref{finite-mean} and Proposition \ref{forwardeqs}.
\end{proof}

\subsection{Proof of Lemmas \ref{asymmetric} and \ref{finite-mean} and of Proposition \ref{forwardeqs}}

We will now compare two zero-range processes starting from elements  $\eta, \xi \in X_f$. 
The particles of $\eta$ will be classified as $\gamma$   and  $\rho$ particles. This means that
at any time $t$ and any site $x$, we will have $\eta_t(x)= \gamma_t(x)+\rho_t(x)$.
This comparison will be made thanks to an auxiliary process $\big((\gamma_t,\rho_t,\xi_t),\,t\ge 0\big)$ on  $X_f^3$. 

Rather than writing down a long generator we state the rates of this process for an arbitrary configuration $(\gamma,\rho,\xi)$. To do so 
we introduce some further notation:
for $\gamma, \rho\in X_f$, define
\begin{equation*}
\rho^{0}(x)=\begin{cases}\rho(0)+1 & x=0, \\ \rho(x)  & x\neq 0, \end{cases}
\end{equation*} 
\begin{equation*}
\gamma^{-1}(x)=\begin{cases}\gamma(-1)-1 & x=-1  \\ \gamma(x)  & x\neq -1, \end{cases}
\end{equation*} 
and
\begin{equation*}
\gamma^{1}(x)=\begin{cases}\gamma(1)-1 & x=1 \\ \gamma(x)  & x\neq 1, \end{cases}
\end{equation*} 
We now state the rates of the auxiliary process,
\begin{enumerate}
\item For $x\in \Z$, at rate $pg\big(\gamma(x)\wedge \xi(x)\big)$ the process jumps to  $(\gamma^{x,x+1},\rho,\xi^{x,x+1})$.
\item For $x\in \Z$, at rate $qg\big(\gamma(x)\wedge \xi(x)\big)$ the process jumps to  $(\gamma^{x,x-1},\rho,\xi^{x,x-1})$.
\item For $x\in \Z$, at rate $p\big[g\big(\xi(x)\big)-g\big(\gamma(x)\wedge \xi(x)\big)\big]$ the process jumps to  $(\gamma,\rho,\xi^{x,x+1})$.
\item For $x\in \Z$, at rate $q\big[g\big(\xi(x)\big)-g\big(\gamma(x)\wedge \xi(x)\big)\big]$ the process jumps to  $(\gamma,\rho,\xi^{x,x-1})$.
\item For $x\in \Z$, at rate $p\big[g\big(\gamma(x)+\rho(x)\big)-g\big(\gamma(x)\big]$ the process jumps to  $(\gamma,\rho^{x,x+1},\xi)$.
\item For $x\in \Z$, at rate $q\big[g\big(\gamma(x)+\rho(x)\big)-g\big(\gamma(x)\big)\big]$ the process jumps to  $(\gamma,\rho^{x,x-1},\xi)$.
\item For $x\in \Z\setminus \{-1\}$, at rate $p\big[g\big(\gamma(x)\big)-g\big(\gamma(x)\wedge \xi(x)\big)\big]$ the process jumps to  
$(\gamma^{x,x+1},\rho,\xi)$.
\item For $x\in \Z\setminus \{1\}$, at rate $q\big[g\big(\gamma(x)\big)-g\big(\gamma(x)\wedge \xi(x)\big)\big]$ the process jumps to  
$(\gamma^{x,x-1},\rho,\xi)$.
\item At rate $p\big[g\big(\gamma(-1)\big)-g\big(\gamma(-1)\wedge \xi(-1)\big)\big]$ the process jumps to $(\gamma^{-1},\rho^0,\xi)$ if $\gamma(0)\geq \xi(0)$,
and to $(\gamma^{-1,0},\rho,\xi) $ if  $\gamma(0)< \xi(0)$.
\item At rate $q\big[g\big(\gamma(1)\big)-g\big(\gamma(1)\wedge \xi(1)\big)\big]$ the process jumps to $(\gamma^{1},\rho^0,\xi)$ if $\gamma(0)\geq \xi(0)$,
and to $(\gamma^{1,0},\rho,\xi) $ if  $\gamma(0)< \xi(0)$.
\end{enumerate}
The processes $(\xi_t,\, t\ge 0)$ and $(\gamma_t,\,t\ge 0)$ follow basic coupling \eqref{basic} except at configurations $\gamma, \xi$ such that $\gamma(-1)>\xi(-1)$ or $\gamma(1)>\xi(1)$, and 
$\gamma(0)\ge \xi(0)$. Then at rate $p\big[g(\gamma(-1))-g(\xi(-1))\big]$ ($q\big[g(\gamma(1))-g(\xi(1))\big]$, resp.) a $\gamma$-particle is removed from $-1$ ($1$, resp.), and added to the $\rho$-process, at $0$.

To derive some of the properties of this process it is convenient to distinguish $\rho$-particles from each other. To do so, we label them with positive integers  and adopt the convention that whenever a $\rho$-particle has to jump from a site $x$ the jump is performed by the 
particle having the lowest label among those present at $x$. If there are $k$ $\rho$-particles in total at time $0$ we label them as $1,\dots,k$ in an arbitrary manner and then each time a $\rho$-particle is created we attribute to it the
 lowest available  label in $\N$. We now let $\Psi(t)$ be the number of $\rho$-particles in the system at time $t$, and denote by $Z_i$  the total number of returns 
 (that is, up to time $\infty$) to the origin of the $i$-th $\rho$-particle. 
We can now state some properties of the process $\big((\gamma_t,\rho_t,\xi_t),\,t\geq 0\big)$.
\begin{enumerate}
\item[{\it i)}] If $\eta_t:=\gamma_t+\rho_t$ (coordinatewise) then $(\eta_t)_{t\geq 0}$ is a zero-range process with rate function $g$.
\item[{\it ii)}]  The process $(\xi_t)_{t\geq 0}$ is a zero-range process with rate function $g$.
\item[{\it iii)}] If $\gamma_0(0)\le \xi_0(0)$ then $\gamma_t(0)\le \xi_t(0)$ for all $t\ge 0$.
\item[{\it iv)}]  Assume $p\neq q$. If the initial configuration is such that  $\rho(x)=0$ for all $x\neq 0$  then the conditional distribution of $Z_1,\dots,Z_k$ given $\{\Psi(t)=k\}$  corresponds to i.i.d.~geometrically distributed random variables.
\end{enumerate}
The first two properties are immediate consequences of the jump rates. Note that they imply that the total number of particles is conserved. Hence, for any initial configuration in $X_f^3$ the jump rates are bounded.  The third property follows from items $9$ and $10$ and the fact that the rates are non-decreasing. The fourth property is a consequence of the following facts, which are derived from the jump rates.
\begin{enumerate}
\item The trajectories of the $\rho$-particles have initial position at the origin.
\item The creation of $\rho$-particles only depends on the evolution of $\gamma$ and $\xi$-particles.
\item Due to the convention adopted for the jumps of the $\rho$-particles, the discrete-time skeletons of the continuous-time trajectories perform i.i.d.~$(p,q)$-random walks,  which are independent of the evolution of the $\gamma,\,\xi$ particles.
\end{enumerate}

Our next lemma follows  from these considerations.

\begin{lemma}\label{countingrho}
 Let $d=1$, $\{g(n)\}_{n\ge 0}$ be as in \eqref{crecientes}, and let $\{p(x,y)\}_{x, y \in \Z}$ be as in \eqref{1dhypo} with $p\neq q$. Assume that at time $0$ there are no $\rho$-particles off the origin and let $H(t)$ be the number of jumps off $0$ performed by $\rho$-particles over the interval $[0,t]$. Then there exists $0<C<\infty$ independent of $t$ such that 
 \[
 E\big[H(t)\big]\leq C E\big[\sum_x \rho_t(x)\big].
 \]
 \end{lemma}
\begin{proof}{Lemma}{countingrho}
Let $X_i$ be the number of jumps off $0$ performed by the i-th $\rho$-particle and let $N(t)$ be the number of $\rho$-particles created in the time interval $[0,t]$. Then $X_i\leq Z_i+1$ for all $i\geq 1$, and since we assume no $\rho$-particles are present at time $0$, $N(t)=\sum_x \rho_t(x)=\psi(t)$. By property iv) above $\psi(t)$ is independent of the random variables $Z_i$. Hence,
\[
E\big[H(t)\big]\leq E\Big[ \sum_{i=1}^{N(t)}(Z_i+1)\Big]=E\big[N(t)\big]\big(E[Z_1]+1\big)
\]
and the lemma follows.
\end{proof}

We recall the mapping $j: X_f^2\to \R$
\[
j(\gamma,\xi)=\Big[\sup_{n\leq 0,m\geq 0}\sum_{x=n}^m (\gamma(x)-\xi(x))\Big]^+ .
\]

\begin{lemma}\label{aux}
Let $\big((\gamma_t,\rho_t,\xi_t),\, t\ge 0\big)$ be the Markov Process on $X_f^3$ with dynamics determined by the rates $1-10$ above.
Denote by $N(t)$ the number of $\xi$-particles jumping off $0$ in the time interval $[0,t]$. Then for all $t\geq 0$
$$\sum_x \rho_t(x) \leq \sum_x \rho_0(x) + N(t) + j(\gamma_0,\xi_0).$$ 
\end{lemma}
\begin{proof}{Lemma}{aux}
First note that $j(\gamma_s,\xi_s)$ can increase by at most one unit at any given time, and that this can only occur when a $\xi$-particle jumps off $0$. We omit the proof of this assertion since it follows the same arguments as the proof of Lemma \ref{l3}. Then, note that $j(\gamma_s,\xi_s)$ decreases by one unit when a $\rho$-particle is created.
Therefore, 
$$\sum_x\rho_t(x) - \sum_x \rho_0(x)\leq N(t)+j(\gamma_0,\xi_0)-j(\gamma_t,\xi_t)$$
and the lemma follows from the fact that $j(\gamma_t,\xi_t)$ is non-negative.
\end{proof}

\begin{proof}{Lemma}{asymmetric}
In this proof we adopt the following notation: $E^{\eta}$ and $E^{\gamma,\rho,\xi}$ will be the expectations associated with the zero-range process starting from $\eta$, and with the auxiliary process starting from $(\gamma,\rho,\xi)$, respectively.
Since the construction is translation invariant, it suffices to show that \eqref{fg1} holds when  $x=0$.\\
 We wish to compare $E^{\eta}\big[\int_0^tg(\eta_s(0))ds\big]$ and  $E^{\xi}\big[\int_0^tg(\eta_s(0))ds\big]$ where $\eta$ and $\xi$ are arbitrary 
 elements of $X_f$. To do so we will apply Lemma \ref{aux} to an  initial configuration $(\gamma,\rho,\xi)$
 such that for all $x\neq 0$,  $\gamma(x)=\eta(x)$ and $\rho(x)=0$,  $\gamma(0)=\eta(0)\wedge \xi(0)$ and $\rho(0)=[\eta(0)-\gamma(0)]^+$.    
As noted before $\gamma_t +\rho_t$ is a zero-range process with initial configuration $\eta$. Hence 
\begin{equation}
\label{boris1}
\begin{split}
&E^{\eta}\big[\#\{\text{$\eta$-particles that have jumped off } 0 \text{ over }[0,t]\}\big]\\
&\qquad=E^{\gamma, \rho,\xi}\big[\#\{\gamma\text{-particles that have jumped off } 0 \text{ over }[0,t]\}\big]\\
&\qquad\quad+E^{\gamma, \rho,\xi}\big[\#\{\rho\text{-particles that have jumped off } 0 \text{ over }[0,t]\}\big].
\end{split}
\end{equation}
Due to property {\it iii)} of the construction, and recalling that  $N(t)$ stands for the number of $\xi$-particles jumping off $0$ in the time interval $[0,t]$, we have
\begin{align*}
& E^{\gamma, \rho,\xi}\big[\#\{\gamma\text{-particles that have jumped off } 0 \text{ over }[0,t]\}\big]\le E^{\gamma, \rho,\xi}\big[N(t)\big].
\end{align*}
From  \eqref{boris1}, the previous inequality, and Lemma \ref{countingrho}, we get
\begin{align*}
&E^{\eta}\big[\#\{\text{$\eta$-particles that have jumped off } 0 \text{ over }[0,t]\}\big] \\
&\hspace{4cm}
\le E^{\gamma, \rho,\xi}\big[N(t) \big]+ CE^{\gamma, \rho,\xi}\big[\sum_x \rho_t(x)],
\end{align*}
and using that in the auxiliary process the $\xi$-particles evolve as in a zero-range process, we conclude
\begin{align}
\label{borisf}
&E^{\eta}\big[\#\{\text{$\eta$-particles that have jumped off } 0 \text{ over }[0,t]\}\big]
\notag\\
&\hspace{4cm}\le E^{\xi}\big[N(t)\big]+CE^{\gamma,\rho,\xi}\big[\sum_x \rho_t(x)].
\end{align}
Let us now fix $\eta \in Y$ and pick $\alpha >0$ and $\beta>0$ such that 
\[
\int \xi(0)d\mu_{\alpha}(\xi) >\beta > \max \Big\{\limsup_n \frac {1}{n}\sum_{x=1}^n \eta(-x), \limsup_n \frac {1}{n}\sum_{x=1}^n \eta(x)\Big\}.
\]
We then apply \eqref{borisf} to $\eta^n$ and a random configuration $\xi$ distributed according to $\Pi_n(\mu_{\alpha})$, where $\Pi_n$ was defined in \eqref{pi_n(mu)}. We obtain
\begin{align}
\label{ukip}
&E^{\eta^n}\big[\#\{\text{$\eta$-particles that have jumped off } 0 \text{ over }[0,t]\}\big] \notag \\
 &\le \int E^{\xi}\big[N(t)\big]d\big(\Pi_n(\mu_{\alpha})\big)(\xi)+\int CE^{\gamma^n,\rho^n,\xi}\big[\sum_x \rho_t(x)\big]d\big(\Pi_n(\mu_{\alpha})\big)(\xi),
\end{align}
where for all $x\neq 0$,  $\gamma^n(x)=\eta^n(x)$ and $\rho^n(x)=0$,  $\gamma^n(0)=\eta^n(0)\wedge \xi(0)$ and $\rho^n(0)=[\eta^n(0)-\gamma^n(0)]$.

The first term of the right hand side above is
equal to
\[
\int E^{\xi}\Big[\int_0^t  g(\xi_s(0))ds \Big]\,d\big(\Pi_n(\mu_{\alpha})\big)(\xi)\le
\int E^{\xi}\Big[\int_0^t  g(\xi_s(0))ds \Big]\,d(\mu_{\alpha})(\xi)
\]
Since $\mu_{\alpha}$ is invariant, this is equal to
\[t \int g(\xi(0))\,d(\mu_{\alpha})(\xi)=t\alpha<\infty.\]
To obtain an upper bound for the second term in \eqref{ukip} we use Lemma \ref{aux} as follows:
\begin{equation}
\label{farage}
\begin{split}
&\int E^{\gamma^n, \rho^n,\xi}\big[\sum_x \rho_t(x)\big]d\big(\Pi_n(\mu_{\alpha})\big)(\xi) \\
&\le \int E^{\gamma^n, \rho^n,\xi}\Big[ \sum_x \rho_0(x)+N(t)+j(\gamma^n,\xi)\Big]d\big(\Pi_n(\mu_{\alpha})\big)(\xi) \\
&\le \int E^{\gamma^n, \rho^n,\xi}\Big[\eta(0)+N(t)+j(\gamma^n,\xi)\Big]\,d\big(\Pi_n(\mu_{\alpha})\big)(\xi) \\
&\le \eta(0)+ \int  E^{\xi}\big[N(t))\big]\,d\big(\Pi_n(\mu_{\alpha})\big)(\xi)+\int  j(\eta^n,\xi)\,d\big(\Pi_n(\mu_{\alpha})\big)(\xi)  \\
&\le \eta(0)+ t\alpha+
\int  j(\eta^n,\xi) d\big(\Pi_n(\mu_{\alpha})\big)(\xi).
\end{split}
\end{equation}
To complete the proof we show that 
$\int  j(\eta^n,\xi) d(\Pi_n(\mu_{\alpha}))(\xi)$ is bounded uniformly in $n$. To do so, note that
\[
 \int  j(\eta^n,\xi) d\big(\Pi_n(\mu_{\alpha})\big)(\xi)\leq  \int  j(\eta,\xi) d(\mu_{\alpha})(\xi).
\]
It now suffices to prove that
$\int  j(\eta,\xi) d(\mu_{\alpha}(\xi))$ is finite.  This is done as follows: let $k$ be such that
\[
\max \Big\{\sup_{m\geq k}\frac{1}{m}\sum_{x=1}^m \eta(-x), \,\sup_{m\geq k}\frac{1}{m}\sum_{x=1}^m \eta(x)\Big\}< \beta.
\]
 Then, define  
\[
 L(\xi)=\inf \Big\{\ell: \inf_{m\geq \ell}\frac{1}{m}\sum_{x=1}^m \xi(-x)\geq \beta,\inf_{m\geq \ell}\frac{1}{m}\sum_{x=1}^m \xi(x)\geq \beta\Big\} .
 \]
If $\xi$ is distributed according to $\mu_{\alpha}$ the random variables $\xi(x), x\in \Z$ are i.i.d.~and by Lemma \ref{inv-m} have finite exponential moments. Hence, we can apply standard large deviation results to conclude that
$\mu_{\alpha}(L(\xi)\geq \ell)$ decays exponentially with $\ell$.  Now define  $M(\xi)=\max \{L(\xi), k\} $. Then
$j(\eta,\xi) \leq \sum_{x=-M(\xi) }^{M(\xi)}\eta(x)$. Hence
\begin{equation*}
\int j(\eta,\xi) d\mu_{\alpha}(\xi)\leq \beta \Big(1 +2\int M(\xi)d\mu_{\alpha}(\xi)\Big)<\infty
\end{equation*}
and the proof is completed.
\end{proof}

The proof of Lemma \ref{asymmetric} relied on the fact that the number of visits to the origin of a random walk with non-vanishing drift 
 has finite expectation. This fails in the symmetric case, and to prove that the conclusion of the lemma still holds in 
this case, we will 
restrict the family of rate functions to 
those having at most exponential growth. 

\begin{proof}{Lemma}{finite-mean} 
It is enough to prove the lemma for $x=0$. Recall the graphical representation from $\S$ \ref{sec:1d}. We will use it to simultaneously construct the zero-range process for all nearest neighbour dynamics determined by 
jump probabilities $(p,q)$.

Fix $n\in \N$ and let $\eta^n$ be the truncated configuration in \eqref{eta-n}.  Under the graphical construction and not counting multiple visits, the number of particles initially to the left of the origin that at any point during $[0,t]$ have reached $0$ is maximal for the totally asymmetric dynamics $(p,q)=(1,0)$, whereas the number of particles initially to the right of the origin that ever reach it during $[0,t]$ is maximal for the opposite totally asymmetric dynamics $(0,1)$. Indeed, enumerate particles in the initial $\eta^n$ configuration according to their distance to the origin and any arbitrary order for particles occupying the same site. Let then $X_{i,n}$ be the position of the $i$-th particle in $\eta^n$ at time $0$, and let us denote by $X^{(p,q)}_{i,n}(t)$ its position at time $t$ under the $(p,q)$ dynamics. Furthermore, let us stipulate that on the event that there is a jump out of a site in the graphical construction, the particle with the highest  (lowest) index at the site is removed if 
the direction of the jump is to the right (left). Then it is easy to 
check that the graphical construction ensures that 
\[
X_{i,n}^{(p,q)}(t)\le X_{i,n}^{(1,0)}(t)\quad \text{ and }\quad X_{i,n}^{(p,q)}(t)\ge X_{i,n}^{(0,1)}(t)\quad  \text{ for all } t\ge 0.
\]
In particular, if a particle initially to the left of the origin ever reached it during $[0,t]$ for the $(p,q)$ dynamics, the same must hold for the $(1,0)$ dynamics, with an analogous statement holding for the particles initially to the right of the origin and the $(0,1)$ dynamics. For the rest of the proof, we continue using the superscript $(p,q)$ to specify which particular dynamics is being referred to.

Fix $t>0$. For the $(p,q)$ dynamics, we have
\begin{align*}
&\hspace{-.7cm}\sharp\big\{ \text{particles that reach $0$ over $[0,t]$} \big\}\le \sharp\big\{i\in \Z,\, X_{i,n}=0\big\} \\
&\hspace{2cm}+\sharp\Big\{i\in \Z,\, X_{i,n}<0,\,X^{(p,q)}_{i,n}(s)\ge 0 \text{ for some } s\in [0,t]\Big\} \\
&\hspace{2cm}+\sharp\Big\{i\in \Z,\, X_{i,n}>0,\,X^{(p,q)}_{i,n}(s)\le 0 \text{ for some } s\in [0,t]\Big\}.
\end{align*}
Let $r\ge 1$. Due to the bound $g(k)\le e^{\lambda k}, k\in \Z$, and the previous observations,
\begin{align*}
&\hspace{-.7cm} g\big(\eta^{n,(p,q)}_s(0)\big)^r\le e^{\lambda r \eta^n(0)}\,e^{\lambda r \,\sharp\big\{ i\in \Z,\,X_{i,n}<0,\,X_{i,n}^{(1,0)}(t)\ge0\big\}}\\
&\hspace{6cm} \times e^{\lambda r\,\sharp\big\{ i\in \Z,\,X_{i,n}>0,\,X_{i,n}^{(1,0)}(t)\le0\big\}},
\end{align*}
uniformly for $s\in [0,t]$. Taking expectations and applying the Cauchy-Schwarz inequality, we get
\begin{align*}
&\hspace{-.9cm}\sup_{s\in [0,t]}\EE^{\eta^n}\Big[ g\big(\eta^{n,(p,q)}_s(0)\big)^r\Big]\le  e^{\lambda r \eta^n(0)}\,
\Big[ \EE^{\eta^n}\Big( e^{2\lambda r \,\sharp\big\{ i\in \Z,\,X_{i,n}<0,\,X_{i,n}^{(1,0)}(t)\ge0\big\}}\Big) \Big]^{\frac{1}{2}}
\\
&\hspace{4cm}\times \Big[ \EE^{\eta^n}\Big( e^{2\lambda r\,\sharp\big\{ i\in \Z,\,X_{i,n}>0,\,X_{i,n}^{(1,0)}(t)\le 0\big\}}\Big) 
\Big]^{\frac{1}{2}}. 
\end{align*}
We need to show that the last two factors on the right above are uniformly bounded in $n$. We treat the first, the proof
for the second one is completely analogous.

To do so, given $(\xi, \xi') \in Y^2$ such that 
\begin{align*}
 j'(\xi,\xi')=\sup_{n\geq 1}\big[ \sum_{x=-n}^{-1}(\xi(x)-\xi'(x))\Big]^+<\infty,
\end{align*}
 we consider the coupled versions of two zero-range processes $(\xi_s,\xi'_s)$ following the $(1,0)$~dynamics, with initial states $(\xi,\xi')$ obtained by means of the graphical representation. Let  $N(t,\xi)$ and $N(t,\xi')$ be the number of particles jumping from $-1$ to $0$ over the time interval $[0,t]$ for the configurations 
$\xi$ and $\xi'$
respectively.
It is now easy to verify that
\begin{align}
\label{eqn2}
 j'(\xi_s,\xi'_s)+N(s,\xi)-N(s,\xi') \le j'(\xi,\xi').
\end{align}
Indeed, $j'$ only increases when $N(s,\xi')$ increases and $N(s,\xi)$ remains constant, and $j'$ decreases whenever $N(s,\xi)$
increases and $N(s,\xi')$ remains constant. But since $j'$ is nonnegative, \eqref{eqn2} implies that
\begin{equation}
\label{eqn3}
 N(t,\xi)\leq N(t,\xi') +j'(\xi,\xi').
\end{equation}
For $\eta \in Y$ the configuration in the statement of the lemma, define
\begin{align}
&\alpha=\max\Big\{ \limsup_{n\in \N}\frac{1}{n}\sum_{x=1}^n \eta(-x),\,\limsup_{n\in \N}\frac{1}{n}\sum_{x=1}^n \eta(x)  \Big\}, \notag%\label{alphasym}
\\
&N_0=\inf\big\{n\in \N,\, \forall  l\ge n: \frac{1}{l} \sum_{x=1}^l \eta(-x) \le \alpha+1 \text{ and } \frac1l\sum_{x=1}^l \eta(x) \le \alpha+1 \big\} \label{n0sym},
\end{align}
and consider $\phi_0>0$ such that $E^{\mu_{\phi_0}}[\eta(0)]=\rho>\alpha+1$. For $\sigma\sim \mu_{\phi_0}$, let 
\[
K(\sigma)=\inf\big\{n\in \N,\,\text{ for all } l\ge n: \sum_{x=1}^l \eta(-x)\le \sum_{x=1}^l \sigma(-x) \big\}<\infty\ \ \text{a.s..}
\]
With this definition we have $j'(\eta,\sigma)\le \sum_{x=-1}^{K(\sigma)} \eta(-x)$. Note moreover that for $\xi \in Y$, $\# \{ i\in \Z,\ X_{i,n}<0, X_{i,n}^{(1,0)}(t)\ge 0\}=N(t,\xi^n)$.
Then, by \eqref{eqn3} applied to $\eta$ and $\sigma$, we get
\begin{equation*}
\EE^{\eta^n}\Big[ e^{2\lambda r\,\sharp\big\{ i\in \Z,\,X_{i,n}<0,\,X_{i,n}^{(1,0)}(t)\ge0\big\}}\Big] \le \EE^{\mu_{\phi_0}}\big[  e^{2\lambda r\sum_{x=1}^{K(\sigma)} \eta(-x)}\,
e^{2\lambda r\, N(t,\sigma) } \big].
\end{equation*}
By Cauchy-Schwarz again,
\begin{align*}
&E^{\eta^n}\Big[ e^{2\lambda r\,\sharp\big\{ i\in \Z,\,X_{i,n}<0,\,X_{i,n}^{(1,0)}(t)\ge0\big\}}\Big]  \\
&\hspace{3.5cm}\le E^{\mu_{\phi_0}}\big[  e^{4\lambda r \sum_{x=1}^{K(\sigma)} \eta(-x)}\big]^{\frac{1}{2}}
\EE^{\mu_{\phi_0}}\big[e^{4\lambda r\, N(t,\sigma) } \big]^{\frac{1}{2}}.
\end{align*}
By the fact that the invariant distribution $\mu_{\phi_0}$ is a product measure, the occupation number at the origin $\{\sigma^{(1,0)}_s(0),\,s\ge 0\}$, $\sigma^{(1,0)}_0\sim \mu_{\phi_0}$, is a birth and death process at equilibrium with constant birth rate 
$\phi_0=E^{\mu_{\phi_0}}[g(\sigma(-1))]$ (and death rate $g(\sigma^{(1,0)}_s(0))$). Therefore $\big(N_t:=N(t,\sigma),\,t\ge 0\big)$ is a Poisson process with intensity $\phi_0$, and finite exponential moment
\begin{equation*}
\EE^{\mu_{\phi_0}}\big[e^{4 r\lambda \, N_t } \big]=e^{\phi_0(e^{4\lambda r}-1)\,t}\,.
\end{equation*}
It remains to prove that 
$ \EE^{\mu_{\phi_0}}\big[  e^{4\lambda r\sum_{x=1}^{K(\sigma)} \eta(-x)}\big]<\infty$.
We can bound
$K(\sigma)\le N_0+\tilde{K}(\sigma)$, if $N_0$ is as in \eqref{n0sym} and we define 
\[
\tilde{K}(\sigma)=\sup\Big\{k\in \N, \sum_{x=1}^k \sigma(-x)\le (\alpha+1)k \Big\}.
\]
Then 
\begin{align*}
\EE^{\mu_{\phi_0}}\big[  e^{4\lambda r\sum_{x=1}^{K(\sigma)} \eta(-x)}\big]&\le \EE^{\mu_{\phi_0}}\big[  e^{4\lambda r\sum_{x=1}^{N_0+\tilde{K}(\sigma)} \eta(-x)}\big]\\
&\le\EE^{\mu_{\phi_0}}\big[ e^{\gamma (N_0+\tilde{K}(\sigma))}\big]\\
&=e^{\gamma N_0}\EE^{\mu_{\phi_0}}\big[ e^{\gamma \tilde{K}(\sigma)}\big]
\end{align*}
with $\gamma:=4\lambda r(\alpha+1)$.
By the inclusion of events 
\[\{\tilde{K}(\sigma)=k\}\subseteq \left\{\sum_{x=1}^k \sigma(-x) \le (\alpha+1) k \right\},\]  it remains  to show that
\begin{equation*}
\sum_k e^{\gamma k}\,\mu_{\phi_0}\Big( \sum_{x=1}^k \sigma(-x) \le (\alpha+1) k\Big)<\infty,
\end{equation*}
but this follows from \eqref{ld} if $\phi_0$ is large enough. 
Hence
\[
\sup_{s\in [0,t]}\EE^{\eta^n}\Big[g\big(\eta^{n,(p,q)}_s(0)\big)^r\Big]
\]
is bounded uniformly in $n$, and the lemma follows by taking the limit $n\to \infty$.

\end{proof}

\begin{proof}{Proposition}{forwardeqs}
Let $f:X\to \R$ be a bounded, local function with support $A$, define $\bar{A}:=\big\{y\in \Z,\, \inf_{x\in A}\|y-x\|\le 1\big\}$.
Hypothesis \eqref{onemore}
and the graphical representation described in $\S$\ref{sec:1d} imply that for any $x\in \Z$ and $t>0$  the process $(\eta_s(x);\,s\ge 0)$ has finitely many jumps over $[0,t+1]$, and 
\[
P^{\eta_0}\big(\{\eta_t(x)\neq \eta_{t-}(x) \text{ or } \eta_t(x)\neq \eta_{t+}(x)\text{ for some } x\in \bar{A}\}\big)=0.
\]
Therefore $P(\lim_{s\rightarrow t} Lf(\eta_s)=Lf(\eta_t))=1\ \ \forall t\geq 0$. Since by \eqref{onemore} the random variables $g(\eta_s(x))$, $s\in [0,t+1]$ are uniformly integrable for any
finite $t$,  $E^{\eta_0}[Lf(\eta_s)]$ is continuous in $s\in [0, t+1]$, and from \eqref{integrated} we conclude that $E^{\eta_0}[f(\eta_t)]$ is differentiable with derivative
\begin{equation*}
\frac{d}{dt}E^{\eta_0}[f(\eta_t)]=E^{\eta_0}\big[Lf(\eta_t)\big],
\end{equation*}
which is the forward equation.
\end{proof}

\section{Alternative construction of the zero-range process\label{sec:alter}}

In this section we provide an alternative construction of the zero-range process again under the assumption that the rate function $g$ is non-decreasing. This construction  is less general than the construction of $\S$\ref{sec:c+im} and $\S$\ref{sec:1d} in the sense that it places more restrictive assumptions on the  rate function $g(n)$, for instance for dimension $d$ and mean-zero, finite-range jump dynamics it requires that $\sup_{1\le k\le n} (g(k)-g(k-1)) \leq C n^a$ with $a < \frac{2}{d}$, see Corollary \ref{finite_m}. On the other hand, it also works on dimensions greater than $1$ and  does not require nearest neighbour jump probabilities.   

We start with some definitions. Let $\{p(x,y)\}_{x,y \in \bb Z^d}$ be a family of translation invariant transition probabilities 
 and consider the continuous-time random walk  $(X_t; \,t \geq 0)$ generated by $\{p(x,y)\}_{(x,y) \in \bb Z^d}$. Let $\bb P^z$ the law of $(X_t; \,t \geq 0)$ with initial condition 
$X_0 = z$. Let 
\[
\tau_0 := \inf \{t \geq 0; X_t =0\}
\]
be the hitting time of the origin by the random walk and let
\begin{equation*}
F_z(t) :=P^z(\tau_0\leq t).
\end{equation*}
Notice that it may happen that $\tau_0 = +\infty$ with positive probability. We recall some notation from $\S$2. Given $\eta\in X$ and $\{x^i\}_{i\in \N}$ an enumeration of the particles of $\eta$, let
\begin{equation}
\label{N_3}
\eta^N(z) := \sum_{i=1}^N \mathbf{1}(x^i = z), \qquad z\in \Z^d
\end{equation}
and
\begin{equation*}
h(n) :=\sup_{1\leq j\leq n}(g(j)-g(j-1)).
\end{equation*}
Finally, for $\eta \in X$, $t\geq 0$ and $z\in \Z^d$ we let
\begin{equation*}
 \overline m_z(t,\eta)=\sum_{i\in \N}F_{x^i-z}(h(i)t).
\end{equation*}

\begin{lemma}
\label{l1.milton}
Let $\eta_0 \in X$ be an initial configuration of particles and let $\{x_0^i\}_{i \in \bb N}$ be an enumeration of the particles of $\eta_0$. Let $t \geq 0$ and $z \in \bb Z^d$.
If $\overline{m}_z(t,\eta_0)$ is finite for any $z \in \bb Z^d$, then  $\eta_s(x)<\infty$ for all $s \in [0,t]$
and $x\in \Z^d$,  and  it satisfies 
\[
\log E [ e^{\theta \eta_s(z)}] \leq (e^\theta-1) \overline{m}_z(t,\eta_0)
\]
for any $s \leq t$, any $\theta>0$ and any $z \in \bb Z^d$. 
\end{lemma}

\noindent\begin{proofmilton}
Note that the sequence of initial configurations $\{\eta_0^N\}_{N \in \bb N}$ is increasing, and therefore one can use the graphical representation to construct a sequence $(\eta_t^N; t \geq 0)$ of zero-range processes with initial conditions $\eta_0^N$, such that $\eta_t^N(x)$ is increasing in $N$ for any $t \geq 0$. Therefore, the limit
\[
\eta_t(x) := \lim_{N \to \infty} \eta_t^N(x)
\]
exists in $[0,\infty]$. By monotonicity this limit is the same as in \eqref{climit}. Our aim is to prove that it is finite. For any $N \in \bb N$ and any $t \geq 0$, $\eta_t^N$ and $\eta_t^{N+1}$ differ in only one site $x_t^{N+1}$ and by only one unit. Conditioned on the trajectory of $(\eta_t^N; t \geq 0)$, the process $(x_t^{N+1}; t \geq 0)$ is a time-inhomogeneous random walk with transition rates
\[
r_t^{N+1} (x,y) = p(y-x) \big[ g(\eta_t^N(x)+1)-g(\eta_t^N(x))\big]
\]
and initial position $x_0^{N+1}$. Since the zero-range process $\eta_t^N$ has exactly $N$ particles,
\[
r_t^{N+1}(x,y) \leq p(y-x) h(N+1).
\]
In particular, $(x_t^N; t \geq 0)$ can be coupled with a random walk $(X_t^N; t \geq 0)$ with transition probabilities $\{p(z)\}_{z \in \bb Z^d}$ starting at $x_0^{N}$ in such a way that both walks visit exactly the same sites in exactly the same order, and such that the walk $(x_t^{N}; t \geq 0)$ always visits sites after the walk $(X_{h(N)t}^N; t \geq 0)$. Moreover, we can define these couplings in such a way that the walks $(X_t^N; t \geq 0)_{N \in \bb N}$ are independent. Define
\[
\tau^N_z:= \inf \{ t \geq 0; X_t^N =z\}, \quad
\tilde{\tau}_z^N := \inf\{t \geq 0; x_t^{N}=z\}
\]
and notice that 
\[
\tilde{\tau}_z^N \geq \frac{\tau^N_z}{h(N)}.
\]
Now we observe that the number of particles at site $z$ at time $t$ is bounded by the number of particles that passed by $z$ up to time $t$. Therefore, for any $z \in \bb Z^d$ and any $N \in \bb N$,
\[
\eta_t^N(z) \leq \sum_{i=1}^N \mathbf{1}(\tilde{\tau}_z^i \leq t) \leq \sum_{i=1}^N \mathbf{1}(\tau_{z}^N \leq h(i) t).
\]
Taking expectations, we see that
\[
E^{\eta_0^N}[\eta_s^N(z)] \leq \sum_{i=1}^N F_{x_0^i-z}(h(i)s) \leq \overline{m}_z(s,\eta_0) \leq \overline{m}_z(t,\eta_0).
\]
Therefore, $(\eta_s; 0 \leq s \leq t)$ is well defined. Notice that $\eta_t$ satisfies
\[
\eta_t(z) \leq \sum_{i \in \bb N} \mathbf{1}(\tau_z^i \leq h(i) t).
\]
The RHS of this estimate is a sum of independent random variables with Bernoulli laws of parameter $p_i = F_{x_0^i-z}(t)$. Therefore,
\[
\log E^{\eta_0}[e^{\theta \eta_t(z)}] \leq \sum_{i \in \bb N} \log (1+p_i(e^\theta-1)) \leq (e^\theta-1) \overline{m}_z(t,\eta_0),
\]
which finishes the proof of the lemma.
\end{proofmilton}

\subsection{The martingale problem and the forward equation}

In this section we show that under the conditions stated in Theorem \ref{t1.milton}, the process constructed in Lemma \ref{l1.milton} satisfies the martingale problem.

\begin{proof}{Theorem}{t1.milton}
Recall that according to Lemma \ref{l1.milton}, the process $(\eta_t; 0 \leq t \leq T)$ is well defined as the increasing limit of the processes $(\eta_t^N; 0 \leq t \leq T)_{N \in \bb N}$, which are the zero-range processes with initial configuration $\eta_0^N$ given by \eqref{N_3}.
Since the process $\eta_t^N$ has a finite number of particles, for any local, bounded function $f: X \to \bb R$,
\[
M_t^N(f) := f(\eta_t^N) -f(\eta^N_0) - \int_0^t L f(\eta_s^N) ds
\]
is a martingale. Since $f$ is bounded, $f(\eta_t^N)- f(\eta_0^N)$ converges {\em a.s.}~to $f(\eta_t)-f(\eta_0)$ and also in $L^p$, for any $p>0$. Let $A \subseteq \bb Z^d$ be the support of $f$, that is, $A$ is the smallest subset of $\bb Z^d$ such that $f(\eta) = f(\xi)$ whenever $\eta$ and $\xi$ agree on $A$, and define $\bar{A}=\{y \in \Z^d,\,p(x-y)>0 \text{ for some }x\in A\}$. Since $f$ is local and $p(\cdot, \cdot)$ is finite range, both $A$ and $\bar{A}$ are finite. We have that
\[
\begin{split}
L f(\eta) 
		&= \sum_{x,y \in A} p(y-x) g(\eta(x)) \big( f(\eta^{x,y}) - f(\eta)\big)\\
		&\quad \quad + \sum_{x \in A, y \notin A} p(y-x) g(\eta(x)) \big( f(\eta- \delta_x) - f(\eta)\big) \\
		& \quad \quad \quad \quad + \sum_{x \in A} \sum_{y \notin A} p(x-y) g(\eta(y)) \big(f(\eta+\delta_x)-f(\eta)\big),
\end{split}
\]
where $\delta_x$ is the configuration with exactly one particle at site $x$ and no particles at other sites. From Lemma \ref{l1.milton},
\[
\log E^{\eta_0^N}[ g(\eta_t^N(z))^p] \leq c^p(e^{p \theta} -1)\overline{m}_z(t,\eta_0)
\]
for any $0 \leq t \leq T$, any $N \in \bb N$ and any $z \in \bb Z^d$. 
From \eqref{As1}, this implies that there exists a constant $C = C(f,c,\theta,p)$ such that
\[
E^{\eta_0^N}[|L f(\eta_t^N)|^p] \leq \exp\big\{C(e^{\theta p}-1) \sup_{z\in \bar{A}} \overline m_z(T,\eta_0)\big\}
\]
for any $0 \leq t \leq T$ and any $N \in \bb N$.
This shows that 
\[
\lim_{N \to \infty} \int_0^t L f(\eta_s^N) ds = \int_0^t L f(\eta_s) ds
\]
in $L^p$ for all $p>1$. Therefore the process
\[
M_t(f) := \lim_{N \to \infty} M_t^N(f)
\]
is well defined and it is a martingale, as we wanted to show. Now \eqref{forward} follows from an argument analogous to the one applied to prove Proposition 
\ref{forwardeqs}.
\end{proof}

\subsection{Finite range random walks}
\label{sec:srw}

\begin{proof}{Corollary}{finite_m}
Since under condition {\em a)} or {\em b)}  the rate function  is bounded by an exponential, we only need to show that \eqref{As1} is satisfied.
First note that under the assumption
\[
\limsup_{m \to \infty} \frac{1}{(2m+1)^d} \sum_{\|x\|\leq m} \eta_0(x) \leq  \rho,
\]
for any $z \in \bb Z^d$ there exists an enumeration of the particles $\{x_0^i\}_{i \in \bb N}$ and a positive constant $c = c(z)$ such that
\begin{equation}
\label{4-1}
\|x_0^k-z\| \geq c k^{1/d}
\end{equation}
for any $k > \eta_0(z)$. Let us obtain a generic bound for the probability $\bb P^x(\tau_0 \leq t)$. By translation invariance,
\[
\bb P^x(\tau_0 \leq t) = \bb P^0(\tau_{-x} \leq t),
\]
where
\[
\tau_{-x} := \inf\{t \geq 0; X_t =-x\},
\]
so we can assume {\em wlog} that $X_0=0$.
We now proceed to show that \eqref{As1} holds under condition {\em a)}. Since
\[
\sum_{z \in \bb Z^d} z p(z) = 0,
\]
the random walk $(X_t; \,t \geq 0)$ is a martingale. Therefore, by Doob's inequality, for any $p >1$,
\[
\bb E^0[ \sup_{0 \leq s \leq t} \|X_s\|^p] \leq C_p \bb E^0[\|X_t\|^p].
\]
Notice that 
\[
\bb P^0(\tau_{-x} \leq t) \leq \bb P^0 \big( \sup_{0 \leq s \leq t} \|X_s\| \geq \|x\|) \leq \frac{C_p \bb E^0[\|X_t\|^p]}{\|x\|^p}.
\]
On the other hand, a simple induction in $p$ argument shows that if $p\in \N$ there exists a constant $C$ depending only on the transition probabilities $\{p(z)\}_{z \in \bb Z^d}$ and the exponent $p$ such that
\[
\bb E^0 [\|X_t\|^p] \leq C t^{p/2}.
\]
This can then be extended to all $p\in \R, p\geq1$ because by Jensen's inequality  for any real valued random variable $Y$, $[E(\vert Y\vert^p)]^{1/p}$ is increasing in $p$.
We conclude that
\begin{equation}
\label{5-1}
\bb P^0(\tau_{-x} \leq t) \leq \frac{C'_p t^{p/2}}{\|x\|^p}.
\end{equation}
Therefore, by \eqref{4-1} and \eqref{5-1},
\[
\sum_{k \in \bb N} F_{x_0^k-z}(h(k) t) \leq \eta_0(z) + \sum_{k > \eta_0(z)} \frac{C t^{p/2} h(k)^{p/2}}{k^{p/d}}.
\]
Taking $p$ large enough to satisfy $a < \frac{2}{d} - \frac{2}{p}$, the above sum converges and  \eqref{As1}  is satisfied.

To prove that condition {\em b)} implies  \eqref{As1} too, first note  that from \eqref{4-1} and the finite range condition on $p(x,y)$, there exists a constant $D$ such that for $k$ large enough, the $k$-th particle  of our enumeration must jump at least $Dk^{1/d}$ times before reaching $z$. Therefore  $F_{x^k_0-z}(h(k)t)$ is bounded above by
\begin{align}
\label{floor}
P\big(Z_1+\dots + Z_{\lfloor Dk^{1/d}\rfloor }\leq h(k)t\big),
\end{align}
where the $Z_i$'s are i.i.d.~exponential random variables of parameter $1$, and $\lfloor Dk^{1/d}\rfloor$ denotes the integer part of $Dk^{1/d}$. Since by condition {\em b)}
\[
\lim_{k\to \infty} \frac{h(k)}{k^{1/d}}=0,
\]
the probability in \eqref{floor} tends to $0$ faster than $\exp(-\alpha k^{1/d})$ for some $\alpha>0$. Therefore
$\sum_k F_{x^k_0-z}(h(k)t)<\infty$
and the proof of the corollary is complete.
\end{proof}

\begin{remark}
\rm 
One could try to use a large deviations bound in order to improve the estimate on the hitting time of the origin. This gives only a marginal improvement on the class of rate functions for which our method applies. The idea is the following. If $h(n) \ll n^{2/d}$, then the event $\{\tau_0 \leq h(n) t\}$ is a large deviation under $\bb P^{x^n_0}$, since the walk does not have time to reach the origin in such a short time. By large deviations, we know how to realize this rare event efficiently: we need to put a constant drift towards the origin, in such a way that the walk arrives to the origin exactly at time $t$. A reasoning along these lines will eventually lead to a sharp estimate, but this sharper estimate does not allow to go beyond $h(n) \ll n^{2/d}$. For $h(n) \gg n^{2/d}$, this is no longer a large deviations estimate, but an estimate on Green's functions, since in that case, up to time $h(n)t$  the walk has time to explore a region of space that has the origin near its center. Therefore, large deviations estimates may help to prove Corollary \ref{finite_m} for a rate $g(n) = \frac{n^{2/d}}{a(n)}$ with $a(n) \uparrow +\infty$ slower than $n^\delta$ for any $\delta >0$, but not to cross the threshold $n^{2/d}$, and fundamentally different techniques are needed to study faster growing rates.
\end{remark}

\begin{remark}
\rm
For finite-range, \emph{asymmetric} random walks, one can apply large deviation bounds to improve condition {\emph b)} in Corollary \ref{finite_m}. The particles that have a reasonable chance to ever hit the origin belong to a curved cone of radius $o(t^{1/2+\varepsilon})$, with vertex at the origin and central ray in the direction opposed to the mean of the walk. The number of points in this cone up to distance $n$ from the origin is of order $o(n^{\frac{d+1}{2} +(d-1)\varepsilon})$. If $\{x_0^n; n \in \mathbb N\}$ is a labeling of the particles in the cone in increasing distance from the origin, we have that $\|x_0^n\| = \mathcal O(n^{\frac{2}{d+1}-\varepsilon'})$ for some $\varepsilon'>0$ that goes to $0$ when $\varepsilon \to 0$. Therefore, Corollary \ref{finite_m} should hold for $a < \frac{2}{d+1}$. Since the computations are very demanding and the result is not optimal, we decided to exclude them from the article.
\end{remark}

\section{Open problems}\label{op}
We finish this paper stating some open problems.
\begin{enumerate}
 \item Does equality hold in item ii) of Proposition \ref{l1}?
 \item By Proposition \ref{l1} we see there is a large class of initial configurations for which the process does not explode. Are all configurations with a finite asymptotic density in that class? Theorem
 \ref{Markov1} states that this is the case when $p(x,y)$ corresponds to a nearest neighbour one-dimensional random walk, but its proof cannot be generalized and new ideas are required.
 \item In the context of Theorem \ref{martingales1}, does the integrated forward equation hold in the symmetric case for any increasing function 
 $g(k)$? Does the backward equation hold if $g(k)$ is bounded by an exponential? Regarding this last question, in \cite{BRSS} the backward equation  is proved  up to some finite time $t$ which depends on the initial configuration  when
 $d=1$ and $p(0,1)=1$.
\end{enumerate}

\section*{Acknowledgments} The authors thank the referees for their careful reading and sharp suggestions. E. Andjel also thanks Maria Eulalia Vares
for interesting discussions concerning the martingale problem and IMPA (Rio de Janeiro, Brazil) for its hospitality. E. Andjel was partially supported first by CNPq grant  300722/2013-3
and then by FAPERJ grant E-26/200.033/2016.
I. Armend\'ariz was partially funded by the PICT grant 2015 "Grafos aleatorios, procesos puntuales y metaestabilidad" and PIP grant PIP11220130100521CO.
M. Jara was funded by ERC Horizon 2020 grant 715734, CNPq  grant  305075/2017-9 and FAPERJ grant E-29/203.012/201. I. Armend\'ariz and M. Jara were partially supported by MathAmSud project "Random Structures and Processes in Statistical Mechanics".

\small 
\bibliographystyle{plain}

\end{document}